\documentclass[a4paper,10pt]{amsart}

\usepackage{amsfonts,amssymb,amscd,amsmath,latexsym,amsbsy,enumerate,stmaryrd,a4wide,verbatim,color}
\usepackage{hyperref}

\theoremstyle{plain}
\newtheorem{theorem}{Theorem}[section]

\newtheorem{lemma}[theorem]{Lemma}
\newtheorem{proposition}[theorem]{Proposition}
\newtheorem{Definition}[theorem]{Definition}
\theoremstyle{remark}
\newtheorem{remark}[theorem]{Remark}
\numberwithin{equation}{section}

\newcommand{\R}{\mathbb R}
\newcommand{\N}{\mathbb N}
\newcommand{\C}{\mathbb C}

\newcommand{\T}{\mathbb T}

\newcommand{\al}{\alpha}
\newcommand{\be}{\beta}
\newcommand{\ga}{\gamma}

\newcommand{\de}{\delta}
\newcommand{\De}{\Delta}

\newcommand{\la}{\lambda}
\newcommand{\La}{\Lambda}

\newcommand{\om}{\omega}

\newcommand{\ba}{\mathbf a}

\newcommand{\bk}{\mathbf k}
\newcommand{\bm}{\mathbf m}
\newcommand{\bn}{\mathbf n}
\newcommand{\bx}{\mathbf x}
\newcommand{\by}{\mathbf y}

\newcommand{\bal}{\boldsymbol{\alpha}}
\newcommand{\bbe}{\boldsymbol{\beta}}
\newcommand{\bnu}{\boldsymbol{\nu}}

\newcommand{\sfY}{\mathsf{Y}}
\newcommand{\sfK}{\mathsf{K}}

\newcommand{\tensor}{\otimes}

\newcommand{\rphis}[5]{\,_{#1}\varphi_{#2} \left( \genfrac{.}{.}{0pt}{}{#3}{#4}
\ ;#5 \right)}
\newcommand{\hf}{\frac{1}{2}}
\newcommand{\su}{\mathfrak{su}}
\newcommand{\U}{\mathcal U}

\newcommand{\mdot}{\,\cdot\,}
\newcommand{\mvert}{\mkern 2mu | \mkern 2mu}

\begin{document}
\title{A quantum algebra approach to multivariate Askey-Wilson polynomials}
\author{Wolter Groenevelt}
\date{\today}
\maketitle

\begin{abstract}
We study matrix elements of a change of base between two different bases of representations of the quantum algebra $\U_q(\su(1,1))$. The two bases, which are multivariate versions of Al-Salam--Chihara polynomials, are eigenfunctions of iterated coproducts of twisted primitive elements. The matrix elements are identified with Gasper and Rahman's multivariate Askey-Wilson polynomials, and from this interpretation we derive their orthogonality relations. Furthermore, the matrix elements are shown to be eigenfunctions of the twisted primitive elements after a change of representation, which gives a quantum algebraic derivation of the fact that the multivariate Askey-Wilson polynomials are solutions of a multivariate bispectral $q$-difference problem.
\end{abstract}

\section{Introduction}

In this paper we give an interpretation of Gasper and Rahman's multivariate Askey-Wilson polynomials \cite{GR05} in representation theory of the quantum algebra $\U_q(\su(1,1))$, and we obtain from this interpretation their main properties: orthogonality relations and difference equations.

The univariate Askey-Wilson polynomials \cite{AW} are orthogonal polynomials depending on four parameters $a,b,c,d$ and on a parameter $q$. They are given explicitly by
\begin{equation} \label{eq:AW pol}
p_n(x;a,b,c,d \mvert q) = \frac{(ab,ac,ad;q)_n }{a^n} \rphis{4}{3}{q^{-n},abcdq^{n-1},ax,a/x}{ab,ac,ad}{q,q},
\end{equation}
where we use standard notation for $q$-shifted factorials and $q$-hypergeometric functions as in \cite{GR}. From the explicit expression \eqref{eq:AW pol} one sees that $p_n(x)$ is a polynomial in $x+x^{-1}$ of degree $n$. The Askey-Wilson polynomials and their discrete counterparts, the $q$-Racah polynomials (which are essentially also Askey-Wilson polynomials), are on top of the Askey-scheme, see \cite{KLS}, a large scheme consisting of families of orthogonal polynomials of ($q$-)hypergeometric type which are related by limit transitions.

The Askey-Wilson polynomials turned out to be fundamental objects in the representation theory of quantum groups and algebras. Koornwinder \cite{K1} gave an interpretation of a two-parameter family of the Askey-Wilson polynomials as zonal spherical functions on the quantum group $\mathrm{SU}_q(2)$. Fundamental in this approach is the introduction of twisted primitive elements, which are elements in the quantum algebra $\U_q(\mathfrak{sl}(2,\C))$ that are much like Lie algebra elements. Similar interpretations for the full four-parameter family of Askey-Wilson polynomials were obtained in e.g.~\cite{Koe},\cite{NM}. A different interpretation is obtained by Rosengren \cite{R}, who introduces a generalized group element (a rediscovery of Babelon's \cite{B} `shifted boundary') that transforms Koornwinder's twisted primitive elements into group-like elements. The Askey-Wilson polynomials appear as `matrix elements' of the generalized group element with respect to continuous and discrete bases in a discrete series representations of the quantum algebra $\U_q(\su(1,1))$. Other interpretations of the Askey-Wilson polynomials, as $3j$ and $6j$-symbols, can be found in e.g.~\cite{KvdJ}, \cite{BR}, \cite{Gr}.

Gasper and Rahman introduced in \cite{GR05} multivariate extensions of the Askey-Wilson polynomials. These polynomials can be considered as $q$-analogues of Tratnik's multivariate Wilson polynomials \cite{Tr}. It should be remarked that the Gasper and Rahman multivariate Askey-Wilson polynomials are different from the Macdonald-Koornwinder polynomials \cite{K}, which are multivariate extensions of Askey-Wilson polynomials as well as extensions of Macdonald polynomials \cite{M} associated to classical root systems.
The Gasper and Rahman multivariate Askey-Wilson polynomials in $d$ variables $x_1+x_1^{-1},\ldots,x_d+x_d^{-1}$ depend, besides $q$, on $d+3$ parameters $\al_0,\ldots,\al_{d+2}$. They can be defined as a nested product of univariate Askey-Wilson polynomials by
\begin{equation} \label{eq:d-var AWpol}
P_d(\bm;\bx;\bal\mvert  q) = \prod_{j=1}^d p_{m_j}\left(x_j;\al_{j} q^{M_{j-1}}, \frac{ \al_j}{\al_0^2} q^{M_{j-1}}, \frac{\al_{j+1}}{\al_j} x_{j+1}, \frac{\al_{j+1}}{\al_j} x_{j+1}^{-1}\mvert q\right),
\end{equation}
where $\bm=(m_1,\ldots,m_d)$, $M_j=\sum_{k=1}^j m_k$, $M_0=0$, $\bal=(\al_0,\ldots, \al_{d+2}) \in \C^{d+3}$, $x_{d+1}=\al_{d+2}$. Under appropriate conditions on the parameters these polynomials are orthogonal on the torus $\T^d$, where $\T$ is the unit circle in the complex plane, with respect to the weight function
\begin{equation} \label{eq:weight multivariate AW}
\frac{1}{(\al_1 x_1^{\pm 1}/\al_0^2, \al_1 x_1^{\pm 1};q)_\infty}\prod_{j=1}^d \frac{ (x_j^{\pm 2};q)_\infty }{ ( \al_{j+1} x_{j+1}^{\pm 1} x_j^{\pm 1}/\al_j;q)_\infty}.
\end{equation}
Here the $\pm$ symbols in the argument of the $q$-shifted factorials means that we take a product over all possible combinations of $+$ and $-$ signs, e.g.~
\[
(ab^{\pm 1} c^{\pm 1};q)_\infty = (abc,ab/c,ac/b,a/bc;q)_\infty.
\]
We will recover the orthogonality relations with respect to \eqref{eq:weight multivariate AW} below.

Iliev \cite{Il} (see also \cite{GI}) showed that the multivariate Askey-Wilson polynomials are eigenfunctions of $d$ commuting difference operators. Furthermore, the multivariate Askey-Wilson polynomials are also eigenfunctions of commuting difference equations in $\bm$, i.e.~they satisfy $d$ independent recurrence relations. In other words, they solve a multivariate bispectral problem in the sense of Duistermaat and Gr\"unbaum \cite{DG}. Below we construct the commuting difference operators in a quantum algebra setting.

Just like the univariate Askey-Wilson polynomials, the multivariate Askey-Wilson polynomials have many families of multivariate orthogonal polynomials and functions as limit cases \cite{GR05}, \cite{GR07}, some of which have found natural interpretations and applications in representation theory of quantum algebras and related physical models:  $2$-variable $q$-Krawtchouk were obtained by Genest, Post and Vinet as matrix elements of $q$-rotations, and they obtained fundamental properties such as orthogonality and difference equations from this interpretation. Rosengren \cite{R01} obtained orthogonality for multivariate $q$-Hahn polynomials from their interpretation as nested Clebsch-Gordan coefficients. Genest, Iliev and Vinet \cite{GIV} obtained from a similar interpretation a difference equation for such $q$-Hahn polynomials, showing that they are wavefunctions for a $q$-deformed quantum Calogero-Gaudin superintegrable systems. Related to this they showed that the multivariate $q$-Racah polynomials appear as $3nj$-coefficients, leading to their orthogonality relation and the duality property. A similar interpretation was also obtained for multivariate $q$-Bessel functions in \cite{Gr}. The multivariate Askey-Wilson polynomials themselves also appear in representation theory: in \cite{KvdJ} Koelink and Van der Jeugt obtained an interpretation of 2-variable Askey-Wilson polynomials as nested Clebsch-Gordan coefficients, and Baseilhac and Martin \cite{BM} constructed infinite dimensional representations of the $q$-Onsager algebra using multivariate Askey-Wilson polynomials and Iliev's corresponding difference operators.

In this paper we extend Rosengren's interpretation of the Askey-Wilson polynomials to a multivariate setting, and in this way we derive the orthogonality relations and the $q$-difference equations for the multivariate Askey-Wilson polynomials. The main ingredients we use are discrete series representations of $\U_q(\su(1,1))$, twisted-primitive elements and properties of univariate Al-Salam--Chihara polynomials. The paper is organized as follows. In Section \ref{sec:Uq} we recall the aspects of representation theory of $\U_q(\su(1,1))$ that we need in this paper; in particular, we give a representation $\pi$ in terms of $q$-difference operators. In Section \ref{sect:Al-Salam--Chihara} we study eigenfunctions of two twisted primitive elements. The eigenfunctions are given in terms of Al-Salam--Chihara polynomials in base $q^2$ and $q^{-2}$. Using properties of these polynomials we introduce two new representations $\rho$ and $\widetilde \rho$, which are equivalent to the representation $\pi$. In Section \ref{sect:AW polynomials} we study the matrix elements for a change of base between two different eigenbases of twisted primitive elements. These matrix elements are essentially (univariate) Askey-Wilson polynomials. We show how the fundamental properties of these polynomials are obtained from this interpretation; the orthogonality relations follow essentially directly from their definition as matrix elements, and the difference equations are shown to correspond to actions of twisted primitive elements in the representations $\rho$ and $\widetilde \rho$. In Sections \ref{sect:multivariate ASC pol} and \ref{sect:multivariate AW pol} we extend the results in the univariate case to the multivariate setting using $N$-fold tensor product representations. In this way we obtain multivariate Askey-Wilson polynomials and their properties from representation theory of the quantum algebra $\U_q(\su(1,1))$. The appendix contains some results on asymptotic behavior of functions we use in this paper, as well as an overview of the various Hilbert spaces appearing in this paper.

\subsection{Notations and conventions}
We assume $0<q<1$, unless explicitly stated otherwise. We denote by $\N$ the set of nonnegative integers, and $\T$ is the unit circle in the complex plane.  For a set $S$, we write $F(S)$ for the vector space consisting of complex valued functions on $S$. If $S$ is countable, we denote by $F_0(S)$ the functions with finite support. By $\mathcal P$ we denote the set of Laurent polynomials in $x_1,\ldots,x_N$ that are invariant under $x_j \leftrightarrow x_j^{-1}$, or equivalently the set of polynomials in $x_j+x_j^{-1}$, $j=1,\ldots,N$ (the number of variables should be clear from the context).

\subsection*{Acknowledgements}
I thank Fokko van de Bult for comments on early versions of this paper.

\section{The quantum algebra $\U_q$} \label{sec:Uq}
The quantum algebra $\mathcal U_q = \U_q\big(\su(1,1)\big)$ is the unital, associative, complex algebra generated by $K$, $K^{-1}$, $E$, and $F$, subject to the relations
\begin{gather*}
K K^{-1} = 1 = K^{-1}K,\\ KE = qEK, \quad KF= q^{-1}FK,\\ EF-FE =\frac{K^2-K^{-2}}{q-q^{-1}}.
\end{gather*}
$\mathcal U_q$ has a $*$-structure $*:\U_q \to \U_q$ and a comultiplication $\De:\mathcal U_q \to \mathcal U_q \tensor \mathcal U_q$ defined
on the generators by
\[
K^*=K, \quad E^*=-F, \quad F^* = -E, \quad (K^{-1})^* = K^{-1},
\]
\begin{equation} \label{eq:comult}
\begin{aligned}
\De(K) &= K \tensor K,& \De(E)&= K \tensor E + E \tensor K^{-1}, \\
\De(K^{-1}) &= K^{-1} \tensor K^{-1},&  \De(F) &= K \tensor F + F \tensor K^{-1}.
\end{aligned}
\end{equation}

\subsection{Twisted primitive elements}
The following two elements of $\mathcal U_q$ play an important role in this paper.
For $s,u \in \C^\times$ the twisted primitive elements $Y_{s,u}$ and $\widetilde Y_{s,u}$ are given by
\begin{equation} \label{eq:def Y tildeY}
\begin{split}
Y_{s,u} &= uq^\hf  EK - u^{-1}q^{-\hf} FK + \mu_s(K^2-1),\\
\widetilde Y_{s,u} &= uq^{-\hf} EK^{-1} - u^{-1}q^\hf FK^{-1} - \mu_s(K^{-2}-1),
\end{split}
\end{equation}
where
\[
\mu_s = \frac{ s + s^{-1}}{q^{-1}-q}.
\]
In particular, we define $Y_s=Y_{s,1}$ and $\widetilde Y_s = \widetilde Y_{s,1}$.
If we formally write $K=q^H$ and $K^{-1}=q^{-H}$, $\widetilde Y_{s,u}$ is obtained from $Y_{s,u}$ by replacing $q$ by $q^{-1}$.
For $s \in \R^\times \cup \T$ and $u \in \T$ both $Y_{s,u}$ and $\widetilde Y_{s,u}$ are self-adjoint in $\U_q$. From \eqref{eq:comult} we find
\begin{equation} \label{eq:comult X}
\begin{split}
\De(Y_{s,u}) &= K^2 \tensor Y_{s,u} + Y_{s,u} \tensor 1, \\
\De(\widetilde Y_{s,u}) &= \widetilde Y_{s,u} \tensor K^{-2} +1 \tensor \widetilde Y_{s,u}.
\end{split}
\end{equation}

\subsection{A representation of $\mathcal U_q$}
Let $k>0$, and let $H=H_{k}$ be the Hilbert space consisting of complex-valued functions on $\N$ with inner product
\begin{gather*}
\langle f,g\rangle_{H} = \sum_{n \in \N} f(n) \overline{g(n)}\, \om(n), \\ \om(n) = \om_{k}(n)=q^{n(k-1)} \frac{ (q^2;q^2)_n }{(q^{2k};q^2)_n}.
\end{gather*}
From the identity $(A;q^{-2})_n = (-A)^n q^{-n(n-1)}(A^{-1};q^{2})_n$ it follows that $\om$, hence also the inner product, is invariant under $q \leftrightarrow q^{-1}$. We consider the following representation $\pi=\pi_k$ on $F(\N)$,
\begin{equation} \label{eq:representation pi}
\begin{split}
[\pi(K) f](n) &= q^{k/2+n} f(n) \\
[\pi(K^{-1}) f](n) &= q^{-k/2-n} f(n) \\
[\pi(E) f](n) & = -\frac{q^{k+n-1}- q^{-k-n+1}}{q^{-1}-q} f(n-1) \\
[\pi(F) f](n) & =  \frac{q^{n+1} - q^{-n-1}}{q^{-1}-q} f(n+1),
\end{split}
\end{equation}
with the convention $f(-1)=0$. This defines an unbounded representation on $H$, where we take $F_0(\N)$ as a dense domain. Furthermore, $\pi$ is a $*$-representation on $H$, i.e.~$\langle \pi(X)f,g \rangle_H = \langle f,\pi(X^*)g\rangle_H$ for $f,g \in F_0(\N)$. Let us remark that if $X^*=X$, then $\pi(X)$ is a symmetric operator, but not necessarily self-adjoint.

\section{Eigenfunctions of twisted primitive elements: Al-Salam--Chihara polynomials} \label{sect:Al-Salam--Chihara}
We determine eigenfunction of the  difference operators $\pi(Y_{s})$ and $\pi(\widetilde Y_{t})$, see also Koelink and Van der Jeugt \cite{KvdJ} and Rosengren \cite{R}. In order to assure self-adjointness of the difference operators corresponding to twisted primitive elements we assume from here on that $s,u \in \T$ and $t \in \R$ such that $|t| \geq q^{-1}$.
We need the (univariate) Al-Salam--Chihara polynomials, see \cite[Section 15.1]{I} and \cite[Section 14.8]{KLS} for details. In this section three different Hilbert spaces $H$, $\mathcal H$ and $\widetilde{\mathcal H}$ are used; for the readers convenience we included a short overview of these Hilbert spaces in the appendix.

\subsection{Al-Salam--Chihara polynomials}
For $q>0$, $q \neq 1$, the Al-Salam--Chihara polynomials are Askey-Wilson polynomials (normalized differently from \eqref{def:AW pol}) with two parameters equal to zero given by
\begin{equation} \label{eq:ASC 3phi2}
\begin{split}
Q_n(x;a,b\mvert q) &= \rphis{3}{2}{q^{-n},ax,a/x}{ab,0}{q,q}\\
& = (ax)^n \frac{(b/x;q)_n}{(ab;q)_n} \rphis{2}{1}{q^{-n},ax}{q^{1-n}x/b}{q,\frac{q}{bx}}.
\end{split}
\end{equation}
They  have the symmetry property
\begin{equation} \label{eq:ACS symmetry}
Q_n(x;b,a\mvert q) = \left(\frac{a}{b}\right)^n Q_n(x;a,b\mvert q).
\end{equation}
The three-term recurrence relation is given by
\[
(x+x^{-1}) Q_n(x) = \tfrac{1}{a}(1-abq^{n})Q_{n+1}(x) +(a+b)q^n Q_n(x) + a(1-q^n) Q_{n-1}(x).
\]
If $0<q<1$, $|a|,|b|<1$ and $\overline{a} = b$ the Al-Salam--Chihara polynomials in base $q$ satisfy the orthogonality relations
\begin{equation} \label{eq:orthogonality ASC}
\begin{gathered}
\frac{1}{4\pi i} \int_\T Q_m(x) Q_n(x) w(x;a,b\mvert q) \frac{dx}{x} = \de_{m,n} \frac{ a^{2n} (q;q)_n}{(ab;q)_n }, \\ w(x;a,b\mvert q) = \frac{(q,ab, x^{\pm 2};q)_\infty }{ (ax^{\pm 1},bx^{\pm 1};q)_\infty},
\end{gathered}
\end{equation}
and the polynomials form a basis for the corresponding weighted $L^2$-space of functions in $x+x^{-1}$.

If $0<q<1$, $ab>1$ and $qb<a$ the Al-Salam--Chihara polynomials in base $q^{-1}$ satisfy the orthogonality relations, see \cite{AI},
\begin{equation} \label{eq:orthogonality q^-1 ASC}
\begin{gathered}
\sum_{y \in aq^{-\N}} Q_n(y;a,b\mvert q^{-1}) Q_n(y;a,b\mvert q^{-1}) W(y;a,b;q) = \de_{m,n} \left(\frac{a}{bq}\right)^n \frac{(q;q)_n}{(1/ab;q)_n}, \\
W(y;a,b;q) = \frac{ 1-q^{2m}/a^2}{1-1/a^2} \frac{ (1/a^2,1/ab;q)_m (bq/a;q)_\infty}{(q,bq/a;q)_m (q/a^{2};q)_\infty} \left( \frac{b}{a} \right)^m q^{m^2}, \qquad y=aq^{-m}.
\end{gathered}
\end{equation}
Here $1/ab = \frac1{ab}$.
Under the conditions above, the $q^{-1}$-Al-Salam--Chihara moment problem is determinate, so the polynomials form a basis for the corresponding weighted $L^2$-space consisting of functions in $y+y^{-1}$. From \eqref{eq:ASC 3phi2} it follows that $Q_n(aq^{-m};a,b\mvert q^{-1})$ is a polynomial in $q^n$ of degree $m$, which can be shown to be a multiple of a little $q$-Jacobi polynomial $p_m(q^n;a^{-1}b, q^{-1}a^{-1}b^{-1};q)$ using $q$-hypergeometric transformations: first transform the $_3\varphi_2$-function in base $q^{-1}$ to a $_3\varphi_1$-function in base $q$, and then transform this into a $_2\varphi_1$-function using \cite[(III.8)]{GR};
\[
Q_n(aq^{-m};a,b\mvert q^{-1}) = \left(-\frac{a}{b}\right)^m  q^{-\frac12 m(m+1)} \frac{(qb/a;q)_m}{(1/ab;q)_m} \rphis{2}{1}{q^{-m},q^m/a^2}{qb/a}{q,q^{1+n}}.
\]
In particular, the dual orthogonality relations
\begin{equation} \label{eq:dual orth rel ASC}
\sum_{n \in \N} Q_n(aq^{-m};a,b\mvert q^{-1}) Q_n(aq^{-r};a,b\mvert q^{-1}) \left(\frac{bq}{a}\right)^n \frac{(1/ab;q)_n}{(q;q)_n} = \frac{ \de_{m,r} }{W(aq^{-m};a,b;q)},
\end{equation}
correspond to the orthogonality relations for the little $q$-Jacobi polynomials.

The following $q$-difference equations will be useful later on.
\begin{lemma}
For $q>0$, $q \neq 1$, the Al-Salam--Chihara polynomials satisfy
\begin{equation} \label{eq:diffeq1}
 Q_n(x;a,b | q)=\frac{1-ax}{1-x^2} Q_n(x q^\hf ;aq^\hf,b/q^{\hf} | q) + \frac{1-a/x}{1-1/x^{2}} Q_n(x/q^\hf;aq^\hf,b/q^{\hf} | q).
\end{equation}
As a consequence, the following $q$-difference equations hold:
\begin{equation} \label{eq:diffeq ACS}
\begin{split}
q^{-n} Q_n(x;a,b\mvert q) & =  \frac{ (1-ax)(1-bx)}{(1-x^2)(1-qx^2)} Q_n(x q;a,b\mvert q)+ \frac{ (1-a/x)(1-b/x)}{(1-1/x^{2})(1-q/x^{2})} Q_n(x/q;a,b\mvert q) \\
&\quad + \frac{ (1+q)(q+ab)-(x+1/x)(aq+bq)}{ (1-qx^2)(1-q/x^{2})} Q_n(x;a,b\mvert q)
\end{split}
\end{equation}
and
\begin{equation} \label{eq:diffeq2 ACS}
\begin{split}
Q_n(x;a,b\mvert q) & = \frac{ (1-ax)(1-aqx)}{(1-x^2)(1-qx^2)}Q_n(x q;aq,b/q\mvert q) + \frac{ (1-a/x)(1-aq/x)}{(1-1/x^{2})(1-q/x^{2})}Q_n(x/q;aq,b/q\mvert q) \\
& \quad + \frac{q (q + 1) (1- ax) (1 - a /x)}{(1-q x^2)(1-q/x^{2})}Q_n(x;aq,b/q\mvert q).
\end{split}
\end{equation}
\end{lemma}
Identity \eqref{eq:diffeq ACS} is the well-known $q$-difference equation for the Al-Salam--Chihara polynomials.
\begin{proof}
Using the explicit expression for the Al-Salam--Chihara polynomials as $_3\varphi_2$-functions, the right hand side of \eqref{eq:diffeq1} equals
\[
\begin{split}
\frac{1-ax}{1-x^2} & \rphis{3}{2}{q^{-n},aqx,a/x}{ab,0}{q,q} + \frac{1-a/x}{1-1/x^{2}} \rphis{3}{2}{q^{-n},ax,aq/x}{ab,0}{q,q} \\
& = \sum_{j=0}^n \frac{ (q^{-n};q)_j q^j }{(q,ab;q)_j }\left( \frac{ (ax;q)_{j+1} (a/x;q)_j}{1-x^2} +  \frac{ (ax;q)_{j} (a/x;q)_{j+1}}{1-1/x^{2}} \right).
\end{split}
\]
The expression in large brackets is equal to $(ax,ax^{-1};q)_j$, so that \eqref{eq:diffeq1} follows.

The difference equation \eqref{eq:diffeq ACS} follows from applying the symmetry \eqref{eq:ACS symmetry} and then applying \eqref{eq:diffeq1} again on the right hand side of \eqref{eq:diffeq1}. The coefficients of $Q_n(x;a,b)$ can be rewritten to the expression in the lemma. Similarly \eqref{eq:diffeq2 ACS} follows from applying \eqref{eq:diffeq1} to itself.
\end{proof}

\subsection{Eigenfunctions of $Y_{s}$}
From \eqref{eq:def Y tildeY} and \eqref{eq:representation pi} it follows that $\pi(Y_{s})$ acts as a three-term difference operator on $F(\N)$ by
\begin{equation} \label{eq:action of Y}
\begin{split}
(q^{-1}-q)[&\pi(Y_{s})f](n) = q^{-(k-1)/2}(1-q^{2k+2n-2}) f(n-1)\\ &  + (s+s^{-1})(q^{k+2n}-1)f(n) + q^{(k-1)/2} (1-q^{2n+2}) f(n+1).
\end{split}
\end{equation}
Using the three-term recurrence relation for Al-Salam--Chihara polynomials we can find eigenfunctions (in the algebraic sense) of $\pi(Y_s)$. We define
\begin{equation} \label{eq:definition vxs}
v_{x,s}(n)=v_{x,s,k}(n)= \left(\frac{q^{-(3k-1)/2}}{s}\right)^n \frac{(q^{2k};q^2)_n }{ (q^2;q^2)_n }Q_n(x;sq^k, s^{-1}q^k\mvert q^2).
\end{equation}
Then $v_{x,s}$ is an eigenfunction of $\pi(Y_{s})$, i.e.
\begin{equation} \label{eq:eigenfunction Xs}
[\pi(Y_{s}) v_{x,s}](n) = \la_{x,s} v_{x,s}(n), \qquad \la_{x,s} = \frac{x+x^{-1}- s- s^{-1}}{q^{-1}-q}= \mu_x - \mu_s.
\end{equation}
Note that $v_{x,s}(n)$ is real-valued for $x,s \in \R^\times\cup\T$.

We are also interested in eigenfunctions of $Y_{s,u}$. These can easily be obtained from eigenfunctions of $Y_s$. Let $M_u$ be the multiplication operator on $F(\N)$ defined by $[M_u f](n)=u^nf(n)$. Then $\pi(Y_{s,u}) M_u = M_u \pi(Y_s)$, and we obtain the following result.
\begin{lemma} \label{lem:eigenfunction Ysu}
For $u \in \T$,
\[
[\pi(Y_{s,u}) M_u v_{x,s}](n) = \la_{x,s}\,  M_u v_{x,s}(n).
\]
\end{lemma}

Let $\mathcal H=\mathcal H_{k,s}$ be the Hilbert space consisting of functions on $\T$ that are $x \leftrightarrow x^{-1}$ invariant almost everywhere, with inner product
\[
\langle f,g\rangle_{\mathcal H} = \frac{1}{4\pi i} \int_\T f(x) \overline{g(x)} \, w(x)\, \frac{dx}{x},
\]
where
\[
w(x)=w_{k,s}(x) = w(x;q^ks, q^k/s \mvert q^2) = \frac{ (q^2,q^{2k},x^{\pm 2};q^2)_\infty}{(q^k s^{\pm 1} x^{\pm 1};q^2)_\infty}.
\]
The set $\{v_{\mdot,s}(n) \mid n \in \N\}$ is an orthogonal basis for $\mathcal H$ with orthogonality relations
\begin{equation} \label{eq:norm v}
\left \langle v_{\mdot,s}(n),v_{\mdot,s}(n') \right\rangle_{\mathcal H} = \frac{\de_{n,n'}}{\om(n)},
\end{equation}
which follows from \eqref{eq:orthogonality ASC}. Note that the squared norm $\om(n)^{-1}$ is independent of $s$.
\begin{proposition} \label{prop:Lambda}
Let $\La=\La_{k,s}:F_0(\N) \to \mathcal P$ be defined by
\[
(\La f)(x) = \langle f, v_{x,s} \rangle_{H},
\]
then $\La$ intertwines $\pi(Y_{s})$ with  multiplication by $\la_{x,s}$. Furthermore, $\La$ extends to unitary operator $H \to \mathcal H$.
\end{proposition}
\begin{proof}
The intertwining property follows from \eqref{eq:eigenfunction Xs} and $Y_{s}^* = Y_{s}$; for $f \in F_0(\N)$,
\[
\big( \La (\pi(Y_{s}) f) \big)(x) = \langle \pi(Y_{s}) f, v_{x,s} \rangle_{H} = \langle f,\pi(Y_{s}) v_{x,s} \rangle_{H} = \la_{x,s} (\La f)(x).
\]
Note that $\La f$ is a finite linear combination of Al-Salam--Chihara polynomials, so $\La f \in \mathcal P$.
For the unitarity, define for $m \in \N$ the function $\de_m \in H$ by $\de_m(n) = \frac{\de_{m,n}}{\om(n)}$, then $\{\de_m\mid m\in \N\}$ is an orthogonal basis for $H$ with squared norm $\|\de_m\|_{H}^2 = (\om(m))^{-1}$. Note that $(\La \de_m)(x) = v_{x,s}(m)$, so that $\La$ maps an orthogonal basis of $H$ to an orthogonal basis of $\mathcal H$ with the same norm, from which it follows that $\La$ extends to a unitary operator.
\end{proof}
Note that it follows from Proposition \ref{prop:Lambda} that $\pi(Y_{s})$ has completely continuous spectrum, which is given by
\[
\{ \la_{x,s} \mid x \in \T \}=\left[-\tfrac{2}{q^{-1}-q}-\mu_s, \tfrac{2}{q^{-1}-q}-\mu_s\right].
\]
\begin{remark} \label{remark}\*
\begin{enumerate}[(i)]
\item The function $v_{x,s}(n)$ is a polynomial of degree $n$ in $x+x^{-1}$. Furthermore, from the explicit expression \eqref{eq:ASC 3phi2} for the Al-Salam--Chihara polynomial and from the symmetry property \eqref{eq:ACS symmetry} it follows that $v_{x,s}(n)$ is also a polynomial in $s+s^{-1}$ of degree $n$.
\item If we assume $s \in \R$ such that $|s|>1$, instead of $s \in\T$, the operator $\pi(Y_{s})$ is still self-adjoint, but now finite discrete spectrum will appear if $|sq^k|>1$. For simplicity we assume $s \in \T$ throughout the paper.
\end{enumerate}
\end{remark}

The set of eigenfunctions $\{v_{x,s} \mid x \in \T\}$ is a generalized basis for $H$ and the set $\{ v_{\cdot,s}(n)\mid n \in \N \}$ is a basis for $\mathcal H$. So using the eigenfunctions $v_{x,s}(n)$, or actually the corresponding operator $\La$, we can transfer the action of $\U_q$ on $H$ to an action on $\mathcal H$. We define a representation $\rho=\rho_{k,s}$ of $\mathcal U_q$ on $\mathcal P$ by
\[
\rho(X) = \La \circ \pi(X) \circ \La^{-1}, \qquad X \in \U_q.
\]
This extends to a $*$-representation on $\mathcal H$.
By Proposition \ref{prop:Lambda} we have an explicit expression for $\rho_k(Y_{s})$ as a multiplication operator. In general it seems very difficult to find explicitly the action of $\rho(X)$ for a given $X \in \U_q$, but for $X = K^{-2}$ we can find such an explicit expression using the difference equation \eqref{eq:diffeq ACS}. We use the following notation for an elementary $q$-difference operator: $[\mathcal T f](x) = f(q^2 x)$.
\begin{lemma} \label{lem:rho(K-2)}
$\rho(K^{-2})$ is the second order $q$-difference operator given by
\[
\rho(K^{-2}) = A(x) \mathcal T + B(x) \mathrm{Id} + A(x^{-1}) \mathcal T^{-1},
\]
where
\[
\begin{split}
A(x) &= A_{k}(x;s)=\frac{q^{-k}(1-q^k sx)(1-q^k x/s)}{(1-x^2)(1-q^2x^2)}, \\
B(x) &= B_{k}(x;s) =\frac{q^{2}(q^{-1}+q)(q^{1-k}+q^{k-1})-q^{2}(x+x^{-1})(s+s^{-1})}{ (1-q^2x^2)(1-q^2/x^{2})}.
\end{split}
\]
\end{lemma}
\begin{proof}
Let $f \in \mathcal P$. From $(K^{-2})^*=K^{-2}$ we obtain
\[
[\rho(K^{-2})f](x) = \big\langle \pi(K^{-2})(\La^{-1} f), v_{x,s} \big\rangle_{H} = \big\langle \La^{-1}f, q^{-k-2(\mdot)} v_{x,s} \big\rangle_{H}.
\]
Now we use the difference equation \eqref{eq:diffeq ACS} with $a=sq^k$, $b=q^k/s$ to rewrite $q^{-k-2n}v_{x,s}(n)$, and the result follows.
\end{proof}

\subsection{Eigenfunctions of $\widetilde Y_t$}
The difference operator $\pi(\widetilde Y_t)$ acts on $f \in F(\N)$ by
\begin{equation} \label{eq:action of tildeY}
\begin{split}
(q -q^{-1})[&\pi(\widetilde Y_{t}) f](n) =q^{(k-1)/2} (1-q^{-2k-2n+2}) f(n-1)\\
&+ (t+t^{-1})(q^{-2n-k}-1) f(n) + q^{(1-k)/2} (1-q^{-2n-2}) f(n+1).
\end{split}
\end{equation}
Note that this is precisely the action of the difference operator $\pi(Y_t)$ with $q$ replaced by $q^{-1}$. So we have the same eigenfunctions, but with $q$ replaced by $q^{-1}$; let
\begin{equation} \label{eq:definition tilde vxs}
\widetilde v_{y,t}(n)= \widetilde v_{y,t,k}(n) = \left(\frac{q^{(3k-1)/2}}{t}\right)^n  \frac{(q^{-2k};q^{-2})_n }{ (q^{-2};q^{-2})_n }Q_n(y;q^{-k}t, q^{-k}/t\mvert q^{-2}),
\end{equation}
then
\[
[\pi(\widetilde Y_t) \widetilde v_{y,t}](n) = \la_{t,y}\, \widetilde v_{y,t}(n).
\]
Note that $\widetilde v_{y,t,k,q}(n) = v_{y,t,k,q^{-1}}(n)$.
Eigenfunctions of $\widetilde Y_{t,u}$, $u \in \T$, are again directly obtained from the eigenfunctions of $\widetilde Y_t$.
\begin{lemma} \label{lem:eigenfunction tildeYtu}
For $u \in \T$,
\[
[\pi(\widetilde Y_{t,u}) M_u \widetilde v_{y,t}](n) = \la_{t,y}\, M_u \widetilde v_{y,t}(n).
\]
\end{lemma}
We define $\widetilde {\mathcal H} = \widetilde{\mathcal H}_{k,t}$ to be the Hilbert space consisting of functions on the set
\[
S=S_{k,t,q} = \{ tq^{-k-2m} \mid m \in \N\}
\]
with inner product
\[
\langle f,g\rangle_{\widetilde{\mathcal H}} = \sum_{y \in S} f(y) \overline{g(y)}\, \widetilde w(y),
\]
where $\widetilde w(y) =\widetilde w_{k,t}(y)$ is the weight function given by
\[
\begin{split}
\widetilde w(y)= W(y;q^{-k}t, q^{-k}/t\mvert q^2)&= \frac{ 1-q^{4m+2k} /t^2 }{1-q^{2k}/t^{2}} \frac{ (q^{2k}/t^2,q^{2k};q^2)_m (q^{2m+2}/t^2;q^2)_\infty }{(q^2;q^2)_m (q^{2k+2}/t^2;q^2)_\infty} t^{-2m}q^{2m^2},
\end{split}
\]
for $y = tq^{-k-2m},  m\in \N$. From the orthogonality relations \eqref{eq:orthogonality q^-1 ASC} for the $q^{-1}$-Al-Salam--Chihara polynomials we find
\begin{equation} \label{eq:norm tilde v}
\left \langle \widetilde v_{\mdot,t}(n),\widetilde v_{\mdot,t}(n') \right \rangle_{\widetilde{\mathcal H}} = \frac{\de_{n,n'}}{\om(n)},
\end{equation}
and the set $\{\widetilde v_{\mdot,t}(n) \mid n \in \N\}$ is an orthogonal basis for $\widetilde{\mathcal H}$. Observe that the squared norm $\om(n)^{-1}$ is independent of $t$.
From the dual orthogonality relations \eqref{eq:dual orth rel ASC} it follows that $\{\widetilde v_{y,t} \mid y \in S \}$ is an orthogonal basis for $H$ with orthogonality relations
\begin{equation} \label{eq:dual orthogality tilde v}
\left \langle  \widetilde v_{y,t},\widetilde v_{y',t} \right \rangle_{H}  = \frac{\de_{y,y'}}{\widetilde w(y)}.
\end{equation}
The proof of the following result is similar to the proof of Proposition \ref{prop:Lambda}.
\begin{proposition} \label{prop:tilde Lambda}
Let $\widetilde \La=\widetilde \La_{k,t}:F_0(\N)\to \mathcal P$ be defined by
\[
(\widetilde\La f)(y) = \langle f, \widetilde v_{y,t} \rangle_{H},
\]
then $\widetilde \La$ intertwines $\pi(\widetilde Y_t)$ with multiplication by $\la_{t,y}$. Furthermore, $\widetilde \La$ extends to a unitary operator $H \to \widetilde{\mathcal H}$.
\end{proposition}
Note that $\pi(\widetilde Y_t)$ has completely discrete spectrum, which is given by
$\{ \la_{t,y} \mid y \in S\}$. \\

Similar as in the previous subsection we define a representation of $\U_q$ on $\mathcal P$ by
\[
\widetilde \rho(X) = \widetilde \La \circ \pi(X) \circ \widetilde \La^{-1}, \qquad X \in \U_q.
\]
This defines a $*$-representation on $\widetilde {\mathcal H}$. In this case $\widetilde \rho(K^2)$ can be given explicitly as a $q$-difference operator using the difference equation \eqref{eq:diffeq ACS} for the Al-Salam--Chihara polynomials. If we denote by $L_{k,s,q}$ the difference operator $\rho(K^{-2})$ given in Lemma \ref{lem:rho(K-2)}, then by construction $\widetilde \rho(K^2)$ is the difference operator $L_{k,t,q^{-1}}$.
\begin{lemma} \label{lem:tilde rho(K2)}
$\widetilde \rho(K^{2})$ is the second order $q$-difference operator given by
\[
\widetilde \rho(K^{2}) = \widetilde A(y)\,\mathcal T^{-1} + \widetilde B(y)\,\mathrm{Id} + \widetilde A(y^{-1})\, \mathcal T,
\]
where $\widetilde A(y)= A_{k,q^{-1}}(y;t)$ and $\widetilde B(y)= B_{k,q^{-1}}(y;t)$.
\end{lemma}
Restricted to the set $S$ the coefficients $\widetilde A$ are given by
\[
\begin{split}
\widetilde A(tq^{-k-2m})& =\frac{q^{k+2} (1-q^{2k+2m}/t^2)(1-q^{2k+2m}) }{t^2(1-q^{2k+4m}/t^2)(1- q^{2k+4m+2}/t^2)},\\
\widetilde A(t^{-1}q^{k+2m})& =\frac{q^k (1-q^{2m}/t^2)(1-q^{2m}) }{(1-q^{2k+4m}/t^2)(1- q^{2k+4m-2}/t^2)}.
\end{split}
\]

\begin{remark}
We can now make the connection with Rosengren's generalized group element, see \cite[\S4.2]{R}. Let $e_m(n)=\de_{mn}$, then the generalized group element is essentially the element $U_{t,u}$ in an appropriate completion of $\U_q$ such that $\pi(U_{t,u}) : F(\N) \to F(\N)$ is given by $\pi(U_{t,u}) e_m = M_u \widetilde v_{tq^{-k-2m},t}$. Then $\pi(U_{t,u})e_m$ is an eigenfunction of $\pi(\widetilde Y_{t,u})$ with eigenvalue $\la_{t,tq^{-k-2m}}$,
so
\[
\pi(U_{t,u}^{-1}\widetilde Y_{t,u} U_{t,u}) e_m = \la_{t,tq^{-k-2m}} e_m,
\]
and more general
\[
\pi\left(U_{t,u}^{-1}X U_{t,u}\right) = \widetilde \rho(X)\Big|_{F(S)},\qquad X \in \U_q,
\]
where we should identity $F(\N) \cong F(S)$. Rosengren's observation that primitive elements are transformed into group-like elements is obtained from $\pi(K^{\pm 2}) e_m = q^{\pm(k+2m)}e_m$, which gives
\[
\begin{split}
\pi\left(U_{t,u}^{-1}\widetilde Y_{t,u} U_{t,u}\right) e_m & = \frac{ t(1-q^{-k-2m}) + t^{-1}(1-q^{k+2m}) }{q^{-1}-q} e_m \\
&= \pi\left( \frac{ t(1-K^{-2}) + t^{-1}(1-K^2) }{ q^{-1}-q } \right) e_m.
\end{split}
\]
Furthermore, Stokman \cite{St} showed that the assignment $X \mapsto U_{t,u}^{-1}\widetilde X U_{t,u}$ transfers the quantum group stucture to a dynamical quantum group structure, so the representations $\rho$ and $\widetilde \rho$ can be considered in the context of dynamical quantum groups. We do not use this connection in this paper, but we can recognize the `dynamical' part in the difference operators in Sections \ref{sect:multivariate ASC pol} and \ref{sect:multivariate AW pol}.
\end{remark}

\section{Eigenfunctions of two twisted primitive elements: Askey-Wilson polynomials} \label{sect:AW polynomials}
In this section we define functions which are eigenfunctions of $\rho(\widetilde Y_{t,u})$ and of $\widetilde \rho(Y_{s,u})$. These functions are multiples of Askey-Wilson polynomials, and we derive properties of the Askey-Wilson polynomials in this way. Moreover, the results in this section serve as a motivation and illustration of the methods we use in Section \ref{sect:multivariate AW pol} to study multivariate Askey-Wilson polynomials.

\bigskip
The functions we study in this section are the matrix elements for a change of base between the discrete basis $\{\widetilde v_{y,t} \mid y \in S\}$ of eigenfunctions of $\widetilde Y_t$ and the continuous basis $\{v_{x,s}\mid x \in \T\}$ of eigenfunctions of $Y_s$.
\begin{Definition} \label{def:AW pol}
For $x \in \T$ and $y \in S_{k,t,q}$, we define
\begin{equation} \label{eq:def univariate AW}
P_{\be}(x,y)  = \langle M_u \widetilde v_{y,t}, v_{x,s} \rangle_{H},
\end{equation}
where $\be$ is the ordered 5-tuple $\be = (s,t,u,k,q)$.
\end{Definition}
Note that $P_{\be}(x,y) = [\La M_u \widetilde v_{y,t}](x) = [\widetilde \La M_{u}v_{x,s}](y)$. It is not a priori clear that the sum \eqref{eq:def univariate AW} converges, since $v_{x,s} \not\in H$. In the appendix it is shown that the sum converges absolutely under the given conditions on $x$ and $y$. We show later on in Lemma \ref{lem:P=AW pol} that $P_\be(x,y)$ is essentially an Askey-Wilson polynomial, and for this reason we will sometimes refer to the functions $P_\be(x,y)$ as Askey-Wilson polynomials (even though they are not polynomials).

We can derive several fundamental properties of the Askey-Wilson polynomials from our definition \eqref{eq:def univariate AW}. We start with the orthogonality relations.
\begin{proposition} \label{prop:orthogonality N=1}
The set $\{P_\be(\mdot,y) \mid y \in S \}$ is an orthogonal basis for $\mathcal H$, with orthogonality relations
\[
\big\langle P_\be(\mdot,y), P_\be(\mdot,y') \big\rangle_{\mathcal H} = \frac{\de_{y,y'}}{\widetilde w(y)}.
\]
\end{proposition}
\begin{proof}
The orthogonality relations and completeness follows from unitarity of $\La$ and $M_u$, and from the orthogonality relations \eqref{eq:dual orthogality tilde v} for $\widetilde v_{y,t}$,
\[
\begin{split}
\left\langle P_\be(\mdot,y), P_\be(\mdot,y')\right\rangle_{\mathcal H} &= \left\langle \La M_u \widetilde v_{y,t},\La M_u \widetilde v_{y',t} \right\rangle_{\mathcal H}  = \left\langle \widetilde v_{y,t}, \widetilde v_{y',t} \right\rangle_{H}  = \frac{\de_{y,y'}}{\widetilde w(y)}. \qedhere
\end{split}
\]
\end{proof}
Our next goal is to obtain difference equations for the Askey-Wilson polynomials $P_\be(x,y)$. Using Lemmas \ref{lem:eigenfunction Ysu} and \ref{lem:eigenfunction tildeYtu} we see that
\begin{equation} \label{eq:AW eigenfunction of tilde Y}
\begin{split}
\big[\rho(\widetilde Y_{t,u}) P_\be(\mdot,y)\big](x) &= \big[\La( \pi(\widetilde Y_{t,u}) M_u \widetilde v_{y,t})\big](x) = \la_{t,y}\, P_\be(x,y), \\
\big[\widetilde \rho(Y_{s,u}) P_\be(x,\mdot)\big](y) &= \big[\widetilde \La( \pi(Y_{s,u}) M_{u} v_{x,s})\big](y) = \la_{x,s}\, P_\be(x,y),
\end{split}
\end{equation}
so $P_\be(x,y)$ is an eigenfunction of $\rho(\widetilde Y_{t,u})$ and also of $\widetilde \rho(Y_{s,u})$. We will show that the eigenfunction equations \eqref{eq:AW eigenfunction of tilde Y} are essentially the difference equation and three-term recurrence relation for the Askey-Wilson polynomials by realizing $\rho(\widetilde Y_{t,u})$ and $\widetilde \rho(Y_{s,u})$ explicitly as difference operators. To do this we express first $\widetilde Y_{t,u}$ in terms of $Y_{s}$ and $K^{-2}$. Similarly, $Y_{s,u}$ can be expresses in terms of $\widetilde Y_{t}$ and $K^2$.
\begin{lemma} \label{lem:Ys S T} \*
\begin{enumerate}[(i)]
\item Let $S,T \in \U_q$ be given by
\[
S = K^{-2} (Y_s  + \mu_s1)-\mu_s1 \quad \text{and} \quad T = \frac{K^{-2}Y_s-Y_sK^{-2}} {q^{-1}-q},
\]
then $S$ and $T$ are independent of $s$, and
\[
\widetilde Y_{t,u} = \frac{(u+u^{-1})S + (qu-q^{-1}u^{-1})T}{q+q^{-1}} + \mu_t(1-K^{-2}).
\]
\item Let $\widetilde S, \widetilde T \in \U_q$ be given by
\[
\widetilde S = K^{2} (\widetilde Y_t  -\mu_t1)+\mu_t1 \quad \text{and} \quad  \widetilde T = \frac{\widetilde Y_tK^{2}- K^{2}\widetilde Y_t} {q^{-1}-q},
\]
then $\widetilde S$ and $\widetilde T$ are independent of $t$, and
\[
Y_{s,u} = \frac{(u+u^{-1})\widetilde S + (q^{-1}u^{-1}-qu)\widetilde T}{q+q^{-1}} - \mu_s(1-K^{2}).
\]
\end{enumerate}
\end{lemma}
\begin{proof}
From the definition of $Y_s$  and the defining relations for $\mathcal U_q$ we find
\[
S = q^{-\frac32}E K^{-1} - q^\frac32F K^{-1}, \qquad T = q^{-\hf}E K^{-1} + q^{\hf}F K^{-1},
\]
which is clearly independent of $s$. Then
\[
q^{-\hf} E K^{-1} = \frac{  S + q T}{q+q^{-1}}, \qquad q^\hf K^{-1}F = \frac{ q^{-1} T - S }{q+q^{-1}}.
\]
Then the result for part (i) follows from the definition of $\widetilde Y_{t,u}$. The proof for part (ii) is similar.
\end{proof}

Using the explicit realizations of $\rho(K^{-2})$ and $\rho(Y_s)$ as difference and multiplication operators from Proposition \ref{prop:Lambda} and Lemma \ref{lem:rho(K-2)} we can now realize $\rho(\widetilde Y_{t,u})$ explicitly as a difference operator. Initially this is a difference operator on $\mathcal P$, and it can be extended to a difference operator acting on meromorphic functions. Identity \eqref{eq:AW eigenfunction of tilde Y} can then be written as a $q$-difference equation for the Askey-Wilson polynomials $P_\be(x,y)$.
\begin{proposition} \label{prop:diffeq AW}
$\rho(\widetilde Y_{t,u})$ is the $q$-difference operator given by
\[
\rho(\widetilde Y_{t,u}) = A_\be(x) \, \mathcal T +  B_\be(x) \, \mathrm{Id} + A_\be(x^{-1}) \mathcal T^{-1},
\]
where
\[
\begin{split}
A_\be(x) &= -A(x)\frac{t(1-qx/ut)(1-u/qtx)}{q^{-1} -q},\\
B_\be(x) &= B(x)\left (\frac{(u+u^{-1})\mu_x}{q^{-1}+q} -\mu_t\right)- \frac{(u+u^{-1})\mu_s}{q^{-1}+q}+\mu_t,
\end{split}
\]
where $A$ and $B$ are given in Lemma \ref{lem:rho(K-2)}.
In particular, $P_\be(x,y)$ satisfies
\[
\la_{t,y}\, P_\be(x,y) =
A_\be(x)P_\be(x q^2,y) +  B_\be(x)P_\be(x,y) + A_\be(x^{-1})P_\be(x/q^2,y).
\]
\end{proposition}
A calculation shows that
\[
B_\be(x) = \frac{tq^{-k}+q^k/t+t+1/t}{q^{-1}-q}  + F(x)+F(x^{-1}),
\]
with
\[
\begin{split}
F(x) & = \frac{tq^{-k}(1-q^k sx)(1-q^kx/s)(1-qux/t)(1-qx/ut)}{(q^{-1}-q)(1-x^2)(1-q^2x^2)} \\
& = -A_\be(x) \frac{ 1-qux/t }{ 1- u/qtx}
\end{split}
\]
\begin{proof}
$\rho(K^{-2})$ is the $q$-difference operator from Lemma \ref{lem:rho(K-2)}, and note that $\rho(Y_s+\mu_s1)$ is multiplication by $\la_{x,s}+\mu_s=\mu_x$. Then $S$ from Lemma \ref{lem:Ys S T} is the $q$-difference operator given by
\[
\rho(S) = \mu_{q^2 x}A(x) \mathcal T +\Big(\mu_x B(x)-\mu_s \Big) \mathrm{Id} +  \mu_{q^{-2}x} A(x^{-1}) \mathcal T^{-1},\\
\]
and $T$ is given by
\[
\rho(T) = \frac{\mu_{q^2x} - \mu_{x}}{q^{-1}-q}A(x) \mathcal T + \frac{\mu_{q^{-2}x} - \mu_{x}}{q^{-1}-q} A(x^{-1}) \mathcal T^{-1}.
\]
Expressing $\widetilde Y_{t,u}$ in terms of $S$ and $T$ using Lemma \ref{lem:Ys S T}, it follows that $\rho(\widetilde Y_{t,u})$ indeed has the form
\[
A_\be(x) \, \mathcal T +  B_\be(x) \, \mathrm{Id} + A_\be(x^{-1}) \mathcal T^{-1}
\]
with $B_\be(x)$ as stated above, and $A_\be(x)$ is given by
\[
A_\be(x)= A(x) \left(\frac{(u+u^{-1})\mu_{q^2 x}}{q^{-1}+q} + \frac{(qu-q^{-1}u^{-1})(\mu_{q^2x}-\mu_x) }{(q^{-1}+q)(q^{-1}-q)} -\mu_t \right).
\]
A small calculation show that the expression between large brackets is equal to
\[
-\frac{t(1-qx/tu)(1-u/qtx)}{q^{-1} -q}. \qedhere
\]
\end{proof}
In a similar way as in Proposition \ref{prop:diffeq AW} we can realize $\widetilde \rho(Y_{s,u})$ as a $q$-difference operator. By construction this operator is obtained from the difference operator $\rho(\widetilde Y_{t,u})$ by replacing $\be=(s,t,u,k,q)$ by $\widetilde \be = (t,s,u,k,q^{-1})$. This immediately leads to a $q$-difference equation in $y$ for $P_\be(x,y)$.
\begin{proposition} \label{prop:diffeq in y AW}
$\widetilde \rho(Y_{s,u})$ is the $q$-difference operator given by
\[
\widetilde \rho(Y_{s,u})= A_{\widetilde \be}(y)\,\mathcal T^{-1} + B_{\widetilde \be}(y)\, \mathrm{Id} + A_{\widetilde \be}(y^{-1})\, \mathcal T.
\]
In particular, for $y \in S$ the Askey-Wilson polynomials $P_\be(x,y)$ satisfy
\[
\la_{x,s}\, P_\be(x,y) = A_{\widetilde \be}(y)P_\be(x,y/q^2) +  B_{\widetilde \be}(y)P_\be(x,y) + A_{\widetilde \be}(y^{-1})P_\be(x,y q^2),
\]
with $P_{\be}(x,tq^{-k+2})=0$.
\end{proposition}

To end this section, let us match the properties of the functions $P_\be(x,y)$ to properties of standard Askey-Wilson polymials defined by \eqref{eq:AW pol}, see \cite{AW}, \cite[Chapter 15]{I}, \cite[Section 14.1]{KLS}. First we show that $P_\be(x,y)$ is a multiple of an Askey-Wilson polynomial, see also \cite[Proposition 4.2]{R}.
\begin{lemma} \label{lem:P=AW pol}
Let
\begin{equation} \label{eq:AW parameters}
(a,b,c,d)= (q^k s, q^k/s, qu/t, q/ut),
\end{equation}
then
\[
P_{\be}(x,y)= (-1)^m d^{-m}q^{-m(m-1)}\frac{(acq^{2m},bcq^{2m};q^2)_\infty }{ (ab;q^2)_m (cx^{\pm 1};q^2)_\infty} p_m(x;a,c,b,d \mvert q^2),
\]
where $y=tq^{-k-2m}$.
\end{lemma}
\begin{proof}
The proof essentially boils down to comparing the recurrence relation from Proposition \ref{prop:diffeq in y AW} with the standard Askey-Wilson recurrence relation, which is
\begin{equation} \label{eq:standard AW rec rel}
\begin{split}
(x+x^{-1}-& a-a^{-1})R_n(x) = \\
&\frac{(1-abq^n)(1-acq^n)(1-adq^n)(1-abcdq^{n-1})} {a(1-abcdq^{2n})(1-abcdq^{2n-1})}\Big( R_{n+1}(x) - R_n(x) \Big) \\
&  + \frac{ a(1-q^n)(1-bcq^{n-1})(1-bdq^{n-1})(1-cdq^{n-1})}{(1-abcdq^{2n-1})(1-abcdq^{2n-2})}\Big( R_{n-1}(x) - R_n(x)\Big),
\end{split}
\end{equation}
with $R_{-1}(x)=0$ and $R_0(x)=1$, where $R_n(x) = a^n p_n(x;a,b,c,d\mvert q) / (ab,ac,ad;q)_n$.
In terms of the Askey-Wilson parameters $(a,b,c,d)$ given by \eqref{eq:AW parameters} the coefficients of the difference operator in Proposition \ref{prop:diffeq in y AW} are
\[
\begin{split}
A^+_m=A_{\widetilde \be}(tq^{-k-2m})&= - \frac{ (1-abq^{2m})(1-acq^{2m})(1-bcq^{2m})(1-abcdq^{2m-2})} {d^{-1}q^{-2m}(q^{-1}-q)(1-abcdq^{4m})(1-abcdq^{4m-2})}, \\
A^-_m = A_{\widetilde \be}(t^{-1}q^{k+2m}) & = -\frac{ (1-q^{2m})(1-adq^{2m-2})(1-bdq^{2m-2}(1-cdq^{2m-2})} {dq^{2m-2}(q^{-1}-q)(1-abcdq^{4m-4})(1-abcdq^{4m-2})},\\
B_m=B_{\widetilde \be}(tq^{-k-2m}) & = \frac{ b + b^{-1} - \sqrt{a/b} - \sqrt{b/a} }{q^{-1}-q} - \frac{A^+_m}{adq^{2m}} \frac{ 1- ad q^{2m} }{1-bcq^{2m-2}} - A^-_m \frac{adq^{2m-2}(1-bcq^{2m})}{ 1- ad q^{2m-2} }
\end{split}
\]
and the functions $P_m(x)=P_\be(x,tq^{-k-2m})$ satisfy the recurrence relation
\[
\la_{x,s} P_m(x) = A_m^+ P_{m+1}(x) + B_m P_m(x) + A_m^- P_{m-1}(x).
\]
Note here that
\[
\la_{x,s} = \frac{ x + x^{-1} - \sqrt{a/b}-\sqrt{b/a} }{q^{-1} - q}.
\]
From \eqref{eq:standard AW rec rel} it follows that the polynomials
\[
\bar R_m(x) = \frac{(-1)^m d^{-m} q^{-m(m-1)}}{(ab,ac,bc;q^2)_m} p_m(x;a,b,c,d \mvert q^2)
\]
are the unique solution for the same recurrence relation with initial value $\bar R_0(x)=1$, so that $P_m(x) = P_0(x) \bar R_m(x)$. It remains to evaluate $P_0(x)=P_\be(x,tq^{-k})$, which is done in the appendix;
\[
P_\be(x,tq^{-k}) = \frac{ (q^{k+1}us^{\pm 1}/t;q^2)_\infty }{ (qux^{\pm 1}/t;q^2)_\infty} = \frac{ (ac,bc;q^2)_\infty }{ (cx^{\pm 1};q^2)_\infty}. \qedhere
\]
\end{proof}
Now we can compare properties of $P_\be(x,y)$ with properties of the Askey-Wilson polynomials $p_m(x;a,b,c,d \mvert q^2 )$. In the proof of Lemma \ref{lem:P=AW pol} we saw that the eigenvalue equation from Proposition \ref{prop:diffeq in y AW} is the three-term recurrence relation. For the difference equation in Proposition \ref{prop:diffeq AW} observe that the coefficients $A_\be$ are given in terms of the Askey-Wilson parameters \eqref{eq:AW parameters} by
\[
A_\be(x) = -\frac{(1-cq^{-2}x^{-1})(1-ax)(1-bx)(1-dx) }{\sqrt{abcdq^{-2}}(q^{-1}-q)(1-x^2)(1-q^2x^2) }.
\]
With this expression it can easily be verified that the difference operator $M^{-1} \circ \rho(\widetilde Y_{t,u}) \circ M$, where $M$ is multiplication by $P_\be(x;tq^{-k})=C/(cx^{\pm 1};q^2)$, is the standard Askey-Wilson difference operator.

For the orthogonality relations observe that $d=\bar c$, and then the orthogonality relations in Proposition \ref{prop:orthogonality N=1} are equivalent to orthogonality relations for $p_n(x;a,b,c,d\mvert q^2)$ with respect to the weight function
\[
\frac{w(x)}{(cx^{\pm 1}, dx^{\pm 1};q^2)_\infty} = \frac{ (x^{\pm 2};q^2)_\infty }{ (ax^{\pm 1}, bx^{\pm 1}, cx^{\pm 1}, dx^{\pm 1};q^2)_\infty },
\]
which is the standard Askey-Wilson weight function.

\begin{remark}
From the recurrence relation it follows that $P_\be(x,y) = P_\be(x;tq^{-k}) p(x)$, $x \in \T$, for some $p \in \mathcal P$. This implies that we can extend $x \mapsto P_\be(x,y)$ to a meromorphic function on $\C$ with poles coming from $P_\be(x;tq^{-k})$. In particular, $P_\be(x,y)$ is also defined for $x = sq^{k+2m'}$, $m \in \N$. Then from the two difference equations we obtain the duality property
\begin{equation} \label{eq:duality}
P_\be(sq^{k+2m'},tq^{-k-2m}) = P_{\widetilde \be}(tq^{-k-2m},sq^{k+2m'}), \qquad m,m'\in \N,
\end{equation}
where $\widetilde \be$ is the ordered 5-tuple $\widetilde \be = (t,s,u,k,q^{-1})$. In case $x=sq^{k+2m'}$, $m' \in \N$, and $|t|> q^{1+k+2m'}$ it is shown in the appendix that the sum \eqref{eq:def univariate AW} converges. In this case the duality property follows directly from Definition \ref{def:AW pol} using $q \leftrightarrow q^{-1}$ invariance of the inner product and $v_{x,s,q^{-1}}(n) = \widetilde v_{x,s,q}(n)$.

Identity \eqref{eq:duality} corresponds to the following identity for Askey-Wilson polynomials,
\[
\frac{p_m(aq^{2m'};a,b,c,d \mvert q^2)}{a^{-m} (ab,ac,ad;q^2)_m} = \frac{ p_{m'}(q^{-2m}/\widetilde a;1/\widetilde a, 1/\widetilde b, 1/\widetilde c, 1/\widetilde d \mvert q^{-2}) }{\widetilde a^{-m'} (1/ab,1/ac,1/ad;q^{-2})_{m'} },
\]
where
\[
\widetilde a = q^{k}/t = \sqrt{abcd/q^2}, \quad \widetilde b = q^k t = ab/\widetilde a, \quad \widetilde c= qus=ac/\widetilde a, \quad \widetilde d = qs/u = ad/\widetilde a.
\]
In terms of $_4\varphi_3$-series this is the identity
\[
\rphis{4}{3}{q^{-m}, q^{-m'}, a_1, a_2}{b_1, b_2, b_3}{q,q} = \rphis{4}{3}{q^m, q^{m'}, 1/a_1, 1/a_2}{1/b_1, 1/b_2, 1/b_3}{1/q,1/q}, \quad b_1b_2b_3q^{m+m'+1} = a_1a_2.
\]
\end{remark}

\section{Multivariate Al-Salam--Chihara polynomials} \label{sect:multivariate ASC pol}
We extend the results from the previous two sections to a multivariate setting by considering tensor product representations. In this section we consider the multivariate analogs of the Al-Salam--Chihara polynomials $v_{x,s}(n)$ and $\widetilde v_{y,t}(n)$. Before we define the representation we are interested in, let us first introduce some convenient notation. We fix a number $N \in \N_{\geq 2}$. For $\ba=(a_1,a_2,\ldots,a_N)$ we denote
\[
\begin{aligned}
\hat \ba &= (a_N,a_{N-1},\ldots,a_1),\\
\ba_j &= (a_1,a_2,\ldots,a_j), & j=1,\ldots,N,\\
q^\ba &= (q^{a_1},q^{a_2},\ldots,q^{a_N}),\\
\Sigma(\ba)&=(a_1,\textstyle{\sum_{j=1}^2 a_j},\ldots,\textstyle{\sum_{j=1}^N a_j}).
\end{aligned}
\]

\bigskip

Let $\bk = (k_1,\ldots,k_N) \in (\R_{>0})^N$. We denote by $\pi$ the representation of $\U_q^{\tensor N}$ on $F(\N^N) \cong F(\N)^{\tensor N}$ given by $\pi=\pi_{k_1} \tensor \cdots \tensor \pi_{k_N}$. This is a $*$-representation on the Hilbert space $H=H_{\bk}$, which is the Hilbert space completion of the algebraic tensor product $H_{k_1} \tensor \cdots \tensor H_{k_N}$. The inner product for $H$ is
\[
\langle f,g\rangle_{H} = \sum_{\bn \in \N^N} f(\bn) \overline{g(\bn)} \, \om(\bn),
\]
with weight function
\[
\om(\bn) = \om_{\bk}(\bn) = \prod_{j=1}^N \om_{k_j}(n_j) = \prod_{j=1}^N \frac{ (q^2;q^2)_{n_j} }{(q^{2k_j};q^2)_{n_j}} q^{n_j(k_j-1)}.
\]
Since the univariate weight function $\om_k(n)$ is $q \leftrightarrow q^{-1}$ invariant, so is the multivariate weight function $\om(\bn)$.

\subsection{Coproducts of twisted primitive elements}
We use the following notation for compositions of coproducts. We define $\De^0$ to be the identity on $\U_q$, and for $n \geq 1$ we define $\De^n:\U_q \to \U_q^{\tensor(n+1)}$ recursively by
\[
\De^n = (\De \tensor 1^{\tensor (n-1)})\De^{n-1}.
\]
Here, and elsewhere, we use the notation $A \tensor B^{\tensor 0}= A$.
Note that we also have
\[
\De^n = (1^{\tensor (n-1)} \tensor \De)\De^{n-1},
\]
which follows from coassociativity of $\De$. A useful property of $\De^n$ is the following one: if $\De(X) = \sum_{(X)} X_{(1)} \tensor X_{(2)}$, then
\begin{equation} \label{eq:De n(X)}
\De^n(X) = \sum_{(X)} \De^{n-m-1}(X_{(1)}) \tensor \De^m(X_{(2)}), \qquad m=0,1,\ldots,n-1.
\end{equation}
This is easily obtained using induction.

We define for $j=1,\ldots, N$ elements $\sfY_{s,u}^{(j)}$ and $\widetilde \sfY_{t,u}^{(j)}$ in $\U_q^{\tensor N}$ by
\[
\begin{split}
\sfY_{s,u}^{(j)} &= 1^{\tensor(N-j)} \tensor \De^{j-1}(Y_{s,u}), \\
\widetilde \sfY_{t,u}^{(j)} &= \De^{j-1}(\widetilde Y_{t,u}) \tensor 1^{\tensor(N-j)},
\end{split}
\]
so essentially these are coproducts of twisted primitive elements in $\U_q^{\tensor N}$.
Similar as before we also define $\sfY_s^{(j)}=\sfY_{s,1}^{(j)}$ and $\widetilde \sfY_t^{(j)}=\widetilde \sfY_{t,1}^{(j)}$. We first show that $\sfY_{s,u}^{(j)}$, $j=1,\ldots,N$, generate a commutative subalgebra of $\U_q^{\tensor N}$, and similarly for $\widetilde{\sfY}_{t,u}^{(j)}$.
\begin{lemma}
For $j,j' = 1,\ldots,N$,
\[
\sfY_{s,u}^{(j)} \sfY_{s,u}^{(j')} = \sfY_{s,u}^{(j')} \sfY_{s,u}^{(j)} \quad \text{and} \quad \widetilde \sfY_{t,u}^{(j)}\widetilde \sfY_{t,u}^{(j')} = \widetilde \sfY_{t,u}^{(j')}\widetilde \sfY_{t,u}^{(j)}.
\]
\end{lemma}
\begin{proof}
We show that $\sfY_{s,u}^{(j)}$ commutes with $\sfY_{s,u}^{(j')}$, where we assume $j'<j$. Note that is suffices to show that $\De^{j-1}(Y_{s,u})$ commutes in $\U_q^{\tensor j}$ with $1^{\tensor(j-j')} \tensor \De^{j'-1}(Y_{s,u})$. Recall from \eqref{eq:comult X} that $\De(Y_{s,u})= K^2 \tensor Y_{s,u} + Y_{s,u}\tensor 1$. Then by \eqref{eq:De n(X)}
\[
\De^{j-1}(Y_{s,u}) = \De^{j-j'-1}(K^2) \tensor \De^{j'-1}(Y_{s,u}) + \De^{j-j'-1}(Y_{s,u}) \tensor \De^{j'-1}(1),
\]
and this clearly commutes with $1^{\tensor(j-j')} \tensor \De^{j'-1}(Y_{s,u})$. The proof for $\widetilde{\sfY}_{t,u}^{(j)}$ is similar.
\end{proof}

In the representation $\pi$ the elements $\sfY_s^{(j)}$ and $\widetilde \sfY_t^{(j)}$ become pairwise commuting difference operators acting on $F(\N^N)$, and we are interested in the common eigenfunctions. First we derive an explicit expression for the difference operators. The following expressions will be useful.
\begin{lemma} \label{lem:Delta j Ys}
For $j=1,\ldots,N$,
\[
\begin{split}
\De^j(Y_{s,u}) & = \sum_{n=0}^j (K^{2})^{\tensor n} \tensor Y_{s,u} \tensor 1^{\tensor(j-n)},\\
\De^j(\widetilde Y_{s,u}) & = \sum_{n=0}^j 1^{\tensor n} \tensor \widetilde Y_{s,u} \tensor (K^{-2})^{\tensor(j-n)}.
\end{split}
\]
\end{lemma}
\begin{proof}
This follows from repeated application of \eqref{eq:De n(X)}, using the coproducts \eqref{eq:comult X} of $Y_{s,u}$ and $\widetilde Y_{s,u}$, and $\De(K^{\pm 2})= K^{\pm 2} \tensor K^{\pm 2}$.
\end{proof}
In order to write down explicit expressions for the difference operators $\pi(\sfY_{s}^{(j)})$ and $\pi(\widetilde{\sfY}_{t}^{(j)})$ we use the following notation for elementary difference operators on $F(\N^N)$:
\[
[T_i^\pm f](\bn) = f(n_1,\ldots,n_{i-1},n_i\pm 1, n_{i+1}, \ldots, n_N).
\]
\begin{proposition}
The difference operator $\pi(\sfY_{s}^{(j)})$ is given by
\[
\pi(\sfY_{s}^{(j)}) = \frac{1}{q^{-1}-q}\sum_{i=N-j+1}^N U_i^{(j),+}(\bn)\, T_i^+ +  U_i^{(j)}\, \mathrm{Id}+ U_i^{(j),-}(\bn)\,T_i^-, \\
\]
where
\begin{align*}
U_i^{(j),+}(\bn) &= q^{\sum_{l=N-j+1}^{i-1}(k_l+2n_l)-(k_i-1)/2} (1-q^{2k_i+2n_i-2}), \\
U_i^{(j),-}(\bn) & = q^{\sum_{l=N-j+1}^{i-1}(k_l+2n_l)+(k_i-1)/2} (1-q^{2n_i+2}),\\
U_i^{(j)}(\bn) & = q^{\sum_{l=N-j+1}^{i-1}(k_l+2n_l)}(s+s^{-1})(q^{2n_i-k_i}-1),
\end{align*}
The difference operator $\pi(\widetilde \sfY_t^{(j)})$ is given by
\[
\pi(\widetilde \sfY_t^{(j)}) = \frac{1}{q-q^{-1}}\sum_{i=1}^j \widetilde U_i^{(j),+}(\bn)\, T_i^+ + \widetilde U_i^{(j)}\, \mathrm{Id} + \widetilde U_i^{(j),-}(\bn)\,T_i^-,
\]
where
\begin{align*}
\widetilde U_i^{(j),+}(\bn) & =  q^{-\sum_{l=i+1}^{j}(k_l+2n_l)-(k_i-1)/2} (1-q^{-2n_i-2}) ,\\
\widetilde U_i^{(j),-}(\bn) &= q^{-\sum_{l=i+1}^j(k_l+2n_l)+(k_i-1)/2} (1-q^{-2k_i-2n_i+2}), \\
\widetilde U_i^{(j)}(\bn) & = q^{-\sum_{l=1}^i(k_l+2n_l)}(t+t^{-1})(q^{-2n_i+k_i}-1).
\end{align*}
\end{proposition}
\begin{proof}
This follows from Lemma \ref{lem:Delta j Ys}, using the actions of $Y_{s}$, $\widetilde Y_{s}$ and $K^{\pm 2}$, see \eqref{eq:action of Y}, \eqref{eq:action of tildeY} and \eqref{eq:representation pi}.
\end{proof}

\subsection{Eigenfunctions}
We write $\bx=(x_1,\ldots,x_N)$, $\by=(y_1,\ldots,y_N)$ and we define $x_{N+1} = s$ and $y_0=t$. We define multivariate analogs of the Al-Salam--Chihara polynomials $v_{x,s}(n)$ in base $q^2$, see \eqref{eq:definition vxs}, and multivariate analogs of the Al-Salam--Chihara polynomials $\widetilde v_{y,t}(n)$ in base $q^{-2}$, see \eqref{eq:definition tilde vxs}, by
\begin{equation} \label{eq:multivariate ASC}
\begin{split}
v_\bx(\bn) &= v_{\bx,s,\bk}(\bn) = \prod_{j=1}^{N} v_{x_j,x_{j+1},k_j}(n_j) \\
 &= \prod_{j=1}^N \left(\frac{q^{-(3k_j-1)/2}}{x_{j+1}}\right)^{n_j} \frac{(q^{2k_j};q^2)_{n_j} }{ (q^2;q^2)_{n_j} }Q_{n_j}(x_j;q^{k_j}x_{j+1}, q^{k_j}/x_{j+1}\mvert q^2), \\
\widetilde v_{\by}(\bn) &= \widetilde v_{\by,t,\bk}(\bn) = \prod_{j=1}^N \widetilde v_{y_j,y_{j-1},k_j}(n_j) \\
&= \prod_{j=1}^N \left( \frac{q^{(3k_j-1)/2}}{y_{j-1}}\right)^{n_j} \frac{(q^{-2k_j};q^{-2})_{n_j} }{ (q^{-2};q^{-2})_{n_j} }Q_{n_j}(y_j;q^{-k}y_{j-1}, q^{-k}/y_{j-1}\mvert q^{-2}).
\end{split}
\end{equation}
Recall from Remark \ref{remark} that $v_{x_j,x_{j+1},k_j}(n_j)$ is a polynomial in $x_j+x_j^{-1}$ and in $x_{j+1}+x_{j+1}^{-1}$, and a similar observation can be made for $\widetilde v_{y_j,y_{j-1},k_j}(n_j)$. So $v_{\bx}(\bn)$ and $\widetilde v_{\by}(\bn)$ are polynomials in $N$ variables. Recall that the univariate Al-Salam--Chihara polynomials $v_{x,s}(n)$ and $\widetilde v_{x,s}(n)$ can be obtained from each other by replacing $q$ by $q^{-1}$. There is a similar relation for their multivariate analogs, which follows directly from \eqref{eq:multivariate ASC}.
\begin{lemma} \label{lem:vq = tilde vq-1}
The multivariate Al-Salam--Chihara polynomials $v_{\bx,s,\bk,q}(\bn)$ and $\widetilde v_{\bx,s,\bk,q}(\bn)$ are related by
\[
v_{\bx,s,\bk,q^{-1}}(\bn) = \widetilde v_{\hat{\bx},s,\hat \bk,q}(\hat\bn).
\]
\end{lemma}

We show that the multivariate Al-Salam--Chihara polynomials are eigenfunctions of the difference operators $\pi(\sfY_{s}^{(j)})$ and $\pi(\widetilde \sfY_t^{(j)})$. For $N=2,3$ this is proved in \cite[Section 4]{KvdJ}. Slightly more general, we will determine eigenfunction of $\pi(\sfY_{s,u}^{(j)})$ and $\pi(\widetilde{\sfY}_{t,u}^{(j)})$. To formulate the result we need the multiplication operator $M_u$ on $F(\N^N)$ defined by $[M_u f](\bn) = u^{n_1+\ldots+n_N} f(\bn)$.
\begin{proposition} \label{prop:multivariate eigenfunctions}
For $j=1,\ldots,N$ and $u \in \T$,
\[
\begin{split}
[\pi(\sfY_{s,u}^{(j)}) M_u v_{\bx}](\bn) &= \la_{x_{N-j+1},s}\, M_u v_{\bx}(\bn),\\
[\pi(\widetilde \sfY_{t,u}^{(j)}) M_u \widetilde v_{\by}](\bn) &= \la_{t,y_j}\, M_u \widetilde v_{\by}(\bn).
\end{split}
\]
\end{proposition}
\begin{proof}
We first set $u=1$ and prove the result for $\widetilde \sfY^{(j)}_{t}=\widetilde \sfY^{(j)}_{t,1}$ using induction on $j$. Let us define for $j=1,\ldots,N$, $\pi_{\bk_j} = \pi_{k_1} \tensor \cdots \tensor \pi_{k_j}$ and
\[
\widetilde v_{\by_j}(\bn_j) = \prod_{i=1}^j \widetilde v_{y_{i},y_{i-1},k_i}(n_i).
\]
Note that $\widetilde v_{\by_{j+1}}(\bn_{j+1}) = \widetilde v_{\by_j}(\bn_j) v_{y_{j+1},y_j,k_{j+1}}(n_{j+1})$. To prove the result for $u=1$ it suffices to prove that
\begin{equation} \label{eq:induction hyp}
\big[\pi_{\bk_j}(\De^{j-1}(\widetilde Y_{t})) \widetilde v_{\by_j}\big](\bn_j) = \la_{t,y_j} \,\widetilde v_{\by_j}(\bn_j),
\end{equation}
for $j=1,\ldots,N$. Then the result follows from the fact that $\widetilde \sfY_t^{(j)}$ acts as $\De^{j-1}(\widetilde Y_t)$ on the first $j$ factors of $F_0(\N)^{\tensor N}$ and as the identity on the other factors.

For $j=1$ identity \eqref{eq:induction hyp} follows directly from Proposition \ref{prop:tilde Lambda}. Assuming \eqref{eq:induction hyp} holds for some $j$, and using
\[
\De^{j}(\widetilde Y_t) = \De^{j-1}(1) \tensor \widetilde Y_t+\De^{j-1}(\widetilde Y_t) \tensor K^{-2},
\]
we find
\[
\begin{split}
\big[&\pi_{\bk_{j+1}}(\De^{j}(\widetilde Y_t)) \widetilde v_{\by_{j+1}}\big](\bn_{j+1}) \\ &= \widetilde v_{\by_j}(\bn_j)  \big[\pi_{k_{j+1}}(\widetilde Y_t) \widetilde v_{y_{j+1},y_{j},k_j}\big](n_{j+1}) + \big[\pi_{\bk_j}(\De^{j-1}(\widetilde Y_t))\widetilde v_{\by_j}\big](\bn_j) \big[\pi_{k_{j+1}}(K^{-2}) \widetilde v_{y_{j+1},y_{j},k_j}\big](n_{j+1})  \\
& = \widetilde v_{\by_j}(\bn_j)  \big[\pi_{k_{j+1}}(\la_{t,y_{j}} K^{-2} + \widetilde Y_t) \widetilde v_{y_{j+1},y_{j},k_j}\big](n_{j+1}).
\end{split}
\]
From $\la_{t,y_{j}} = \mu_{t}-\mu_{y_{j}}$ we find
\[
\la_{t,y_{j}}K^{-2}  + \widetilde Y_t = q^{-\hf} EK - q^{\hf} FK - \mu_{y_{j}} K^{-2} + \mu_t1 = \widetilde Y_{y_{j}}+ (\mu_t-\mu_{y_j})1.
\]
Then using $\pi_k(\widetilde Y_{y_j}) \widetilde v_{y_{j+1},y_j} = \la_{y_j,y_{j+1}} \widetilde v_{y_{j+1},y_j}$, we obtain \eqref{eq:induction hyp} for $j+1$. By induction it follows that $\pi(\widetilde \sfY_t^{(j)}) \widetilde v_\by = \la_{t,y_j} \widetilde v_{\by}$. Finally, from the identities $\pi_k(\widetilde Y_{t,u}) M_u = M_u \pi_k(\widetilde Y_t)$ and $\pi_k(K^{-2})M_u=M_u\pi(K^{-2})$ on $F(\N)$, and the second identity in Lemma \ref{lem:Delta j Ys}, it follows that $\pi(\widetilde \sfY_{t,u}^{(j)}) M_u =M_u \pi(\widetilde \sfY_t^{(j)})$ on $F(\N^N)$, which proves the result for $\widetilde \sfY_{t,u}^{(j)}$. For $\sfY^{(j)}_{s,u}$ the proof runs along the same lines.
\end{proof}

Next we define the corresponding Hilbert spaces. We define the weight function $w=w_{\bk,s}$ on $\T^N$ by
\begin{equation} \label{eq:def multivariate w}
w(\bx) = \prod_{j=1}^N w_{k_j,x_{j+1}}(x_{j}) = C_{\bk} \prod_{j=1}^N \frac{ (x_j^{\pm 2};q^2)_\infty}{(q^{k_j}x_{j+1}^{\pm 1} x_j^{\pm 1};q^2)_\infty},
\end{equation}
where $C_\bk$ is the $\bx$-independent constant given by $C_{\bk} = \prod_{j=1}^N (q^2,q^{2k_j};q^2)_\infty$. The Hilbert space $\mathcal H=\mathcal H_{\bk,s}$ consists of functions on $\T^N$ which are invariant under $x_j \leftrightarrow x_j^{-1}$ for $j=1,\ldots,N$, and which have finite norm with respect to the inner product
\[
\langle f,g \rangle_{\mathcal H} = \int_{\T^N} f(\bx) \overline{g(\bx)} \, w(\bx)\, \frac{d\bx}{\bx},
\]
where $\frac{d\bx}{\bx} = \frac{1}{(4\pi i)^N} \frac{dx_1}{x_1} \cdots \frac{dx_N}{x_N}$. The other Hilbert space consist functions on
the set
\begin{equation} \label{eq:def set S}
S=S_{\bk,t,q} = \Big\{ tq^{-\Sigma(\bk)-2\Sigma(\bm)} \mid \bm \in \N^N \Big\}
\end{equation}
Let $\widetilde w= \widetilde w_{\bk,t}$ be the weight function on $S$ given by
\[
\begin{split}
\widetilde w(\by) = \prod_{j=1}^N &\widetilde w_{k_j,y_{j-1}}(y_j) =
 \prod_{j=1}^N \left(\frac{ q^{2K_{j-1}+4M_{j-1}+2m_j}}{t^2} \right)^{m_j}
 \frac{1-q^{2K_j+4M_j}/t^2 }{ 1-q^{2K_j+4M_{j-1}}/t^2} \\
  & \qquad \times \frac{ (q^{2K_{j}+4M_{j-1}}/t^2,q^{2k_j};q^2)_{m_j} (q^{2K_{j-1}+4M_{j-1}+2m_j+2}/t^2;q^2)_\infty }{(q^2;q^2)_{m_j} (q^{2K_{j}+4M_{j-1}+2}/t^2;q^2)_\infty},
\end{split}
\]
where $\by=tq^{-\Sigma(\bk)-2\Sigma(\bm)}$, $M_j=\Sigma(\bm)_j = \sum_{i=1}^j m_i$, $K_j = \Sigma(\bk)_j = \sum_{i=1}^j k_i$ and $M_0=0=K_0$. The Hilbert space $\widetilde{\mathcal H}= \widetilde{\mathcal H}_{\bk,t}$ consists of functions on $S$ which have finite norm with respect to the inner product
\[
\langle f,g\rangle_{\widetilde{\mathcal H}} = \sum_{\by \in S} f(\by) \overline{g(\by)}\, \widetilde w(\by).
\]

From the orthogonality relations of the univariate polynomials we obtain the following orthogonality relations for the multivariate ones.
\begin{lemma} \label{lem:orth multivariate} \*
\begin{enumerate}[(i)]
\item The set $\{v_{\bullet}(\bn) \mid \bn \in \N^N\}$ is an orthogonal basis for $\mathcal H$ with orthogonality relations
\[
\langle v_{\bullet}(\bn),  v_{\bullet}(\bn')\rangle_{\mathcal H} =  \frac{\de_{\bn,\bn'}}{ \om(\bn)}.
\]
\item The set $\{\widetilde v_{\bullet}(\bn) \mid \bn \in \N^N\}$ is an orthogonal basis for $\widetilde {\mathcal H}$ with orthogonality relations
\[
\langle \widetilde v_{\bullet}(\bn), \widetilde v_{\bullet}(\bn')\rangle_{\widetilde {\mathcal H}} = \frac{\de_{\bn,\bn'}}{\om(\bn)}.
\]
\item The set $\{\widetilde v_\by \mid \by \in S\}$ is an orthogonal basis for $H$ with orthogonality relations
\[
\langle \widetilde v_{\by}, \widetilde v_{\by'} \rangle_H = \frac{ \de_{\by,\by'} }{\widetilde w(\by)}.
\]
\end{enumerate}
\end{lemma}
\begin{proof}
Statement (i) follows from \eqref{eq:norm v} by writing the inner product as an iterated integral $\int_{x_1} \int_{x_2} \cdots \int_{x_N}$ and using that the squared norms of the univariate Al-Salam--Chihara polynomials $v_{x_j,x_{j+1}}(n_j)$ is independent of $x_{j+1}$.
Statement (ii) is proved in a similar way. Statement (iii) is obtained directly from \eqref{eq:dual orthogality tilde v} by taking products.
\end{proof}
\begin{remark}
The orthogonality in statement (iii) of the functions $\widetilde v_{\by}$ with respect to the inner product on $H$ can be considered as orthogonality relations for multivariate little $q$-Jacobi polynomials.
\end{remark}

From Proposition \ref{prop:multivariate eigenfunctions} and Lemma \ref{lem:orth multivariate} we obtain the multivariate analogue of Propositions \ref{prop:Lambda} and \ref{prop:tilde Lambda}.
\begin{proposition} \label{prop:La(N)}\*
\begin{enumerate}[(i)]
\item Define $\La=\La_{\bk,s}:F_0(\N^N) \to \mathcal P$ by
\[
(\La f)(\bx) = \langle f, v_{\bx} \rangle_{H},
\]
then $\La$ intertwines $\pi(\sfY_{s}^{(j)})$, $j=1,\ldots,N$, with multiplication by $\la_{x_{N+1-j},s}$, and extends to a unitary operator $H \to \mathcal H$.
\item Define $\widetilde \La=\widetilde \La_{\bk,t}:F_0(\N^N) \to \mathcal P$ by
\[
(\widetilde\La f)(\by) = \langle f, \widetilde v_{\by} \rangle_{H},
\]
then $\widetilde\La$ intertwines $\pi(\widetilde\sfY_t^{(j)})$, $j=1,\ldots,N$, with multiplication by $\la_{t,y_j}$, and extends to a unitary operator $H \to \widetilde{\mathcal H}$.
\end{enumerate}
\end{proposition}
Next we define representations on $\mathcal P$ using the above defined intertwining operators $\La$ and $\widetilde \La$.

\subsection{Representations on $\mathcal P$}
We define a representation $\rho=\rho_{\bk}$ of $\U_q^{\tensor N}$ on $\mathcal P$ by
\[
\rho(X) = \La  \circ \pi(X) \circ \La^{-1}, \qquad X \in \U_q^{\tensor N}.
\]
In this representation $\sfY_{s}^{(j)}$, $j=1,\ldots,N$, act as multiplication by $\la_{x_{N-j+1},s}$. We define for $j=1,\ldots,N$
\[
\sfK^{-2,(j)} = \De^{j-1}(K^{-2}) \tensor 1^{\tensor(N-j)} \in \U_q^{\tensor N}.
\]
Our next goal is to realize $\rho(\sfK^{-2,(j)})$ as explicit difference operators. To do this we first show that $\pi_k(K^{-2})$ acts on a univariate polynomial $v_{x,s}$ as a `dynamical' difference operator, i.e.~it acts as a difference operator in the variable $x$ and in the parameter $s$.
\begin{lemma} \label{lem:rho(K-2) dynamic}
The univariate Al-Salam--Chihara polynomials $v_{x,s}(n)$ satisfy
\[
\begin{split}
[\pi_k(K^{-2}) v_{x,s}] (n) & =  C(x;s) v_{x q^2,s q^2}(n) + D(x;s) v_{x,s q^2}(n) + C(x^{-1};s) v_{x/q^2,s q^2}(n)\\
& = C(x;s^{-1}) v_{x q^2,s/q^2}(n) + D(x;s^{-1}) v_{x,s/q^2}(n) + C(x^{-1};s^{-1}) v_{x/q^2,s/q^2}(n),
\end{split}
\]
where
\[
\begin{split}
C(x;s) = C_{k}(x;s) & = \frac{q^{-k} (1-q^k x s)(1-q^{k+2}x s)}{(1-x^2)(1-q^2x^2)},\\
D(x;s) = D_{k}(x;s) & = \frac{q^{3-k} (q^{-1} + q) (1- q^ks x) (1 - q^k s/x)}{ (1-q^2 x^2)(1-q^2/x^{2})}.
\end{split}
\]
\end{lemma}
\begin{proof}
Note that $\pi(K^{-2})$ acts as multiplication by $q^{-k-2n}$ on $F(\N)$.
Then the first identity follows from the definition \eqref{eq:definition vxs} of $v_{x,s}$ and from applying Lemma \ref{eq:diffeq2 ACS} with $a=q^ks$, $b=q^k/s$. The second identity follows from the first one using $s \leftrightarrow s^{-1}$ invariance of $v_{x,s}$.
\end{proof}

We are now ready to realize $\pi(\sfK^{-2,(j)})$ as an explicit $q$-difference operator. We use the following notation. For $i=1,\ldots,N$ let $\mathcal T_i$ be the elementary $q$-difference operator defined by
\[
[\mathcal T_i f](\bx) = f(x_1,\ldots,x_{i-1},x_i q^2,x_{i+1},\ldots,x_N).
\]
For $j \in \{1,\ldots, N\}$ and $\bnu = (\nu_1,\ldots,\nu_j) \in \{-1,0,1\}^j$ we define
\[
\mathcal T_{\bnu} = \mathcal T_{1}^{\nu_1} \cdots \mathcal T_{j}^{\nu_j}.
\]
\begin{proposition} \label{prop:K-2 mult difference operator}
For $j=1,\ldots,N$ $\rho(\sfK^{-2,(j)})$ is the $q$-difference operator on $\mathcal P$ given by
\[
\rho(\sfK^{-2,(j)}) = \sum_{\bnu \in \{-1,0,1\}^j } V_{\bnu}^{(j)}(\bx) \mathcal T_{\bnu}
\]
where
\[
V_{\bnu}^{(j)}(\bx) =\prod_{i=1}^j V_{\bnu,i}^{(j)}(\bx),
\]
with
\begin{align*}
V_{\bnu,j}^{(j)}(\bx)=V_{\nu_j,k_j}^{(j)}(x_j;x_{j+1}) & =
\begin{cases}
A_{k_j}(x_j^{\nu_j} ;x_{j+1}), & \nu_j \neq 0,\\
B_{k_j}(x_j;x_{j+1}), & \nu_j=0,
\end{cases} \\
\intertext{and for $i=1,\ldots,j-1$}
V_{\bnu,i}^{(j)}(\bx)=V_{\nu_i,\nu_{i+1},k_i}^{(j)}(x_i;x_{i+1}) & =
\begin{cases}
A_{k_i}(x_i^{\nu_i};x_{i+1}), & \nu_i \neq 0,\, \nu_{i+1}=0,\\
B_{k_i}(x_i;x_{i+1}), & \nu_i=0,\, \nu_{i+1}=0,\\
C_{k_i}(x_i^{\nu_i};x_{i+1}^{\nu_{i+1}}), & \nu_i\neq 0,\, \nu_{i+1}\neq 0,\\
D_{k_i}(x_i;x_{i+1}^{\nu_{i+1}}), & \nu_i= 0,\, \nu_{i+1}\neq 0.
\end{cases}
\end{align*}
Here $A$ and $B$ are given in Lemma \ref{lem:rho(K-2)}, and $C$ and $D$ are given in Lemma \ref{lem:rho(K-2) dynamic}.
In particular, the multivariate Al-Salam--Chihara polynomials $v_{\bx}(\bn)$ satisfy
\[
\sum_{\bnu \in \{-1,0,1\}^j } V_{\bnu}^{(j)}(\bx) [\mathcal T_{\bnu} v_{\bullet}(\bn)](\bx) = q^{-\sum_{i=1}^j (k_i+2n_i)} \, v_{\bx}(\bn), \qquad j=1,\ldots,N.
\]
\end{proposition}
\begin{proof}
Let us fix a number $j$. It suffices to prove the $q$-difference equations for the multivariate Al-Salam--Chihara polynomials, which boils down to proving the identities
\[
\big[\pi_{\bk_j}(\De^{j-1}(K^{-2})) v_{\bx_j}\big](\bn_j)  =  \sum_{\bnu \in \{-1,0,1\}^j } V_{\bnu}^{(j)}(\bx) [\mathcal T_{\bnu}v_{\bx_j}](\bn_j).
\]
Here we use notations similar as in the proof of Proposition \ref{prop:multivariate eigenfunctions}; in particular
\[
v_{\bx_j}(\bn_j) = \prod_{i=1}^j v_{x_i,x_{i+1},k_i}(n_i).
\]
We will show that, for $l=1,\ldots,j$,
\begin{equation} \label{eq:l difference operators}
\begin{split}
\pi_{\bk_j}(\De^{j-1}(K^{-2})) v_{\bx_j} &=\prod_{i=1}^{l-1}  \pi_{k_i}(K^{-2})v_{x_i,x_{i+1},k_i} \\
& \times \sum_{(\nu_{l},\ldots,\nu_j) \in \{-1,0,1\}^{j-l+1} }  V_{(\nu_{l},\ldots,\nu_j)}^{(j)}(\bx) \mathcal T_{(\nu_{l},\ldots,\nu_j)} \prod_{i=l}^j v_{x_i,x_{i+1},k_i},
\end{split}
\end{equation}
where we use the notations
\[
V_{(\nu_l,\ldots,\nu_j)}^{(j)}(\bx) = \prod_{i=l}^j V_{\bnu,i}^{(j)}(\bx) \quad \text{and} \quad \mathcal T_{(\nu_l,\ldots,\nu_j)} = \mathcal T_l^{\nu_l} \cdots \mathcal T_j^{\nu_j}.
\]
For $l=1$ this is the desired result.

We prove \eqref{eq:l difference operators} by backwards induction on $l$. For $l=j$ the identity to prove is
\[
\begin{split}
\pi_{\bk_j}&(\De^{j-1}(K^{-2})) \, v_{\bx_j} = \prod_{i=1}^{j-1}  \pi_{k_i}(K^{-2})\,v_{x_i,x_{i+1},k_i} \\
&  \times \Big( A_{k_j}(x_j;x_{j+1}) v_{ x_j q^2,x_{j+1},k_j} + B_{k_j}(x_j;x_{j+1}) v_{x_j,x_{j+1},k_j} + A_{k_j}(x_j^{-1};x_{j+1}) v_{x_j/q^2,x_{j+1},k_j}\Big),
\end{split}
\]
which is valid by Lemma \ref{lem:rho(K-2)}. Next assume \eqref{eq:l difference operators} holds for some $l$, then
\[
\begin{split}
\pi_{\bk_j}(\De^{j-1}(K^{-2}))\, v_{\bx_j} = & \,F(\bx) \, \pi_{k_{l-1}}(K^{-2})v_{x_{l-1},x_{l},k_{l-1}} \\
&\times \sum_{\nu_{l} \in \{-1,0,1\} }  V_{\nu_{l},\nu_{l+1},k_{l}}^{(j)}(x_{l};x_{l+1})  \mathcal T_{l}^{\nu_{l}} v_{x_l,x_{l+1},k_l},
\end{split}
\]
where $F(\bx)$ is a function (which can be made explicit) independent of $x_{l}$. We rewrite the factor $\pi_{k_{l-1}}(K^{-2}) v_{x_{l-1},x_{l},k_{l-1}}$ depending on the value of $\nu_{l}$: for $\nu_{l}=0$ we use Lemma \ref{lem:rho(K-2)},
for $\nu_{l}=1$ we use the first identity in Lemma \ref{lem:rho(K-2) dynamic}, and for $\nu_{l}=-1$ we use the second identity in Lemma \ref{lem:rho(K-2) dynamic}. This leads to \eqref{eq:l difference operators} for $l-1$.
\end{proof}

We define another representation $\widetilde \rho=\widetilde \rho_{\bk}$ of $\U_q^{\tensor N}$ on $\mathcal P$ by
\[
\widetilde \rho(X) = \widetilde \La \circ \pi(X) \circ \widetilde \La^{-1}, \qquad X \in \mathcal U_q^{\tensor N}.
\]
By Proposition \ref{prop:La(N)} the operator $\widetilde \rho(\widetilde{\sfY}_{t,u}^{(j)})$ is multiplication by $\la_{t,y_j}$. We also define for $j=1,\ldots,N$,
\[
\sfK^{2,(j)} = 1^{\tensor(N-j)} \tensor \De^{j-1}(K^2) \in \U_q^{\tensor N}.
\]
$\widetilde \rho (\sfK^{2,(j)})$ can be realized as an explicit $q$-difference operator on $\mathcal P$. This is done in the same way as we did above for $\rho(\sfK^{-2,(j)})$, and formally it is just replacing $q$ by $q^{-1}$. We omit most of the details. First we need the analog of Lemma \ref{lem:rho(K-2) dynamic} for $\widetilde v_{y,t}$.
\begin{lemma} \label{lem:tilde rho(K2) dynamic}
The univariate Al-Salam--Chihara polynomials $\widetilde v_{y,t}(n)$ satisfy
\[
\begin{split}
[\pi_k(K^{2})\widetilde v_{y,t}] (n) & =  \widetilde C(y;t) \widetilde v_{y/ q^2,t/ q^2}(n) + \widetilde D(y;t) \widetilde v_{y,t/ q^2}(n) + \widetilde C(y^{-1};t) \widetilde v_{yq^2,t/ q^2}(n)\\
& = \widetilde C(y;t^{-1}) \widetilde v_{y/ q^2,tq^2}(n) + \widetilde D(y;t^{-1}) \widetilde v_{y,tq^2}(n) + \widetilde C(y^{-1};t^{-1}) \widetilde v_{yq^2,tq^2}(n),
\end{split}
\]
where
\[
\begin{split}
\widetilde C(y;t) = \widetilde C_{k}(y;t) & = \frac{q^{k} (1-q^{-k} ty)(1-q^{-k-2}t y)}{(1-y^2)(1-y^2/ q^2)},\\
\widetilde D(y;t) = \widetilde D_{k}(y;t) & = \frac{q^{k-3} (q^{-1} + q) (1- q^{-k}ty) (1 - q^{-k} t/y)}{ (1-y^2/q^2)(1-1/x^{2}q^2)}.
\end{split}
\]
\end{lemma}
With this lemma we can prove the analog of Proposition \ref{prop:K-2 mult difference operator}. The following notation will be useful. For $j \in \{1,\ldots, N\}$ and $\bnu = (\nu_1,\ldots,\nu_j) \in \{-1,0,1\}^j$ we define
\[
\hat{\mathcal T}_{\bnu} = \mathcal T_{N}^{\nu_1} \cdots \mathcal T_{N-j+1}^{\nu_j}.
\]
\begin{proposition} \label{prop:K2 mult difference operator}
For $j=1,\ldots,N$ $\widetilde \rho(\sfK^{2,(j)})$ is the $q$-difference operator on $\mathcal P$ given by
\[
\widetilde \rho(\sfK^{2,(j)}) = \sum_{\bnu \in \{-1,0,1\}^j } \widetilde V_{\bnu}^{(j)}(\by) \hat{\mathcal T}_{\bnu}
\]
where
\[
\widetilde V_{\bnu}^{(j)}(\by) =\prod_{i=1}^j \widetilde V_{\bnu,i}^{(j)}(\by),
\]
with
\begin{align*}
\widetilde V_{\bnu,j}^{(j)}(\by)=\widetilde V_{\nu_j,k_{N-j+1}}^{(j)}(y_{N-j+1};y_{N-j}) & =
\begin{cases}
\widetilde A_{k_{N-j+1}}(y_{N-j+1}^{\nu_j} ;y_{N-j}), & \nu_j \neq 0,\\
\widetilde B_{k_{N-j+1}}(y_{N-j+1};y_{N-j}), & \nu_j=0,
\end{cases} \\
\intertext{and for $i=1,\ldots,j-1$}
\widetilde V_{\bnu,i}^{(j)}(\bx)=\widetilde V_{\nu_i,\nu_{i+1},k_{N-i+1}}^{(j)}(y_{N-i+1};y_{N-i}) & =
\begin{cases}
\widetilde A_{k_{N-i+1}}(y_{N-i+1}^{\nu_i};y_{N-i}), & \nu_i \neq 0,\, \nu_{i+1}=0,\\
\widetilde B_{k_{N-i+1}}(y_{N-i+1};y_{N-i}), & \nu_i=0,\, \nu_{i+1}=0,\\
\widetilde C_{k_{N-i+1}}(y_{N-i+1}^{\nu_i};y_{N-i}^{\nu_{i+1}}), & \nu_i\neq 0,\, \nu_{i+1}\neq 0,\\
\widetilde D_{k_{N-i+1}}(y_{N-i+1};y_{N-i}^{\nu_{i+1}}), & \nu_i= 0,\, \nu_{i+1}\neq 0.
\end{cases}
\end{align*}
Here $\widetilde A$ and $\widetilde B$ are given in Lemma \ref{lem:tilde rho(K2)}, and $\widetilde C$ and $\widetilde D$ are given in Lemma \ref{lem:tilde rho(K2) dynamic}.
In particular, the multivariate Al-Salam--Chihara polynomials $\widetilde v_{\by}(\bn)$ satisfy
\[
\sum_{\bnu \in \{-1,0,1\}^j } \widetilde V_{\bnu}^{(j)}(\bx) [\hat{\mathcal T}_{\bnu} v_{\bullet}(\bn)](\bx) = q^{\sum_{i=N-j+1}^N (k_i+2n_i)} \, v_{\by}(\bn), \qquad j=1,\ldots,N.
\]
\end{proposition}

\section{Multivariate Askey-Wilson polynomials} \label{sect:multivariate AW pol}
In this section we define functions which are multivariate extensions of the Askey-Wilson polynomials defined by \eqref{eq:def univariate AW}. Similarly as in Section \ref{sect:AW polynomials} we derive their main properties: orthogonality relations and difference equations. We will also identify them with multiples of the Gasper and Rahman multivariate Askey-Wilson polynomials \eqref{eq:d-var AWpol}.

\bigskip

Similar as in Section \ref{sect:AW polynomials} we study the matrix elements of the change of base between the discrete basis $\{ \widetilde{v}_{\by} \mid \by \in S\}$ of eigenfunctions of $\widetilde{\sfY}_{t}^{(j)}$, $j=1,\ldots,N$, and the continuous basis $\{v_\bx \mid \bx \in \T^N\}$ of eigenfunctions of $\sfY_s^{(j)}$, $j=1,\ldots,N$.
\begin{Definition}
For $\bx \in \T^N$ and $\by \in S= S_{\bk,t,q}$ (see \eqref{eq:def set S} for the set $S_{\bk,t,q}$), we define
\[
P_{\bbe}(\bx,\by) = \langle M_u \widetilde v_{\by}, v_{\bx} \rangle_{H},
\]
where $\bbe$ is the ordered $(N+4)$-tuple given by $\bbe= (s,t,u,k_1,\ldots,k_N,q)$.
\end{Definition}
Observe that $P_{\bbe}(\bx,\by) = \La(M_u \widetilde v_{\by})(\bx) = \widetilde \La(M_u v_{\bx})(\by)$. We first show that these are multiples of the multivariate Askey-Wilson polynomials defined by \eqref{eq:d-var AWpol}.
\begin{theorem} \label{thm:Pbeta=multivariate AWpol}
For $\by = tq^{-\Sigma(\bk)-2\Sigma(\bm)}\in S$,
\[
P_{\bbe}(\bx,\by)  = C_{\bbe}(\bx,\by) P_N(\bm;\bx;\bal\,|\, q^2),
\]
where
\[
C_{\bbe}(\bx,\by) = \frac{(\al_{N+1} \al_{N+2} q^{2M_N},\al_{N+1}q^{2M_N}/\al_{N+2};q^2)_\infty }{(\al_1 x_1^{\pm 1};q^2)_\infty}\prod_{j=1}^N \frac{ \left(- \frac{  \al_0^2q^{-2M_{j-1}}}{\al_j}\right)^{m_j} q^{-m_j(m_j-1)} }{(\al_{j+1}^2/\al_j^2;q^2)_{m_j}},
\]
with
\begin{equation} \label{eq:parameter alpha}
\al_0=u, \qquad \al_{N+2}=s, \qquad \al_j = uq^{K_{j-1}+1}/t \quad \text{for } j=1,\ldots,N+1,
\end{equation}
$M_j = \Sigma(\bm)_j$, $K_j = \Sigma(\bk)_j$ and $M_0=0=K_0$.
\end{theorem}
\begin{proof}
From the definition \eqref{eq:multivariate ASC} of the multivariate Al-Salam--Chihara polynomials and our definition \eqref{eq:def univariate AW} of the univariate Askey-Wilson polynomials we obtain
\[
\begin{split}
P_{\bbe}(\bx,\by) &= \prod_{j=1}^N  \langle M_u \widetilde v_{y_j,y_{j-1},k_j}, v_{x_j,x_{j+1},k_j} \rangle_{H_{k_j}} = \prod_{j=1}^N P_{\be_j}(x_j,y_j),
\end{split}
\]
where $\be_j = (x_{j+1},y_{j-1},u,k_j,q)$. Recall here that $x_{N+1}=s$ and $y_0=t$, and note that $y_j = y_{j-1} q^{-k_j-2m_j}$. By Lemma \ref{lem:P=AW pol} and the symmetry of the Askey-Wilson polynomials $p_n(x;a,b,c,d\mvert q)$ in its parameters $a,b,c,d$, a factor $P_{\be_j}(x_j,y_j)$ is a multiple of the Askey-Wilson polynomial
\[
p_{m_j}\left(x_j;\frac{u}{t}q^{1+K_{j-1}+2M_{j-1}},\frac{1}{ut}q^{1+K_{j-1}+2M_{j-1}}, x_{j+1}q^{k_j},\frac{q^{k_j}}{x_{j+1}}\,|\,q^2\right),
\]
which is the $j$-th factor of the multivariate Askey-Wilson polynomial $P_N(\bm;\bx;\bal\,|\, q^2)$ as defined in \eqref{eq:d-var AWpol}. The expression for $C_{\bbe}(\bx,\by)$ follows from the factor in front of the Askey-Wilson polynomial in Lemma \ref{lem:P=AW pol}, i.e.~
\[
C_{\bbe}(\bx,\by) = \prod_{j=1}^N (-ut)^{m_j} q^{-m_j(1+K_{j-1}+2M_{j-1})} q^{-m_j(m_j-1)} \frac{ (u x_{j+1}^{\pm 1}q^{1+K_j+2M_j}/t;q^2)_\infty }{(q^{2k_j};q^2)_{m_j} (u x_{j}^{\pm 1}q^{1+K_{j-1}+2M_{j-1}}/t;q^2)_\infty }.
\]
This simplifies to the expression given in the theorem by cancelling common factors.
\end{proof}
Next we derive properties of the functions $P_{\bbe}(\bx,\by)$. We start with orthogonality.

\begin{theorem} \label{thm:orthogonality}
The set $\{P_{\bbe}(\mdot,\by) \mid \by \in S \}$ is an orthogonal basis for $\mathcal H$, with orthogonality relations
\[
\left\langle P_{\bbe}(\mdot,\by), P_{\bbe}(\mdot,\by') \right\rangle_{\mathcal H} = \frac{\de_{\by,\by'}}{\widetilde w(\by)}.
\]
\end{theorem}
\begin{proof}
The proof is essentially the same as the proof of Proposition \ref{prop:orthogonality N=1}.
The orthogonality relations follow from $P_{\bbe}(\bx,\by) =  [\La (M_u \widetilde v_{\by}) ](\bx)$, the orthogonality relations of $\widetilde v_{\by}$ in Lemma \ref{lem:orth multivariate}, and unitarity of $\La$ and $M_u$.
\end{proof}
\begin{remark}
From Theorem \ref{thm:Pbeta=multivariate AWpol} it follows that the orthogonality relations from Theorem \ref{thm:orthogonality} are equivalent to orthogonality relations of the multivariate Askey-Wilson polynomials $P_N(\bm;\bx;\bal)$ with respect to the weight function
\[
\frac{w(x)}{(\al_1 x_1^{\pm 1}, \overline{\al_1} x_1^{\pm 1};q^2)_\infty},
\]
where $w$ is defined by \eqref{eq:def multivariate w}. Up to a multiplicative constant $w$ is equal to
\[
\prod_{j=1}^N \frac{ (x_j^{\pm 2};q^2)_\infty }{ ( \al_{j+1}x_{j+1}^{\pm 1} x_j^{\pm 1}/\al_j;q^2)_\infty},
\]
and $\overline{\al_1} = \al_1/\al_0^2$, so we recover the orthogonality relations of $P_N(\bm;\bx;\bal)$ with respect to the weight function \eqref{eq:weight multivariate AW} in base $q^2$.
\end{remark}

The multivariate Askey-Wilson polynomials $P_{\bbe}(\bx,\by)$ are simultaneous eigenfunctions of the coproducts of the twisted primitive elements, i.e., for $j=1,\ldots,N$,
\[
\begin{split}
[\rho(\widetilde \sfY_{t,u}^{(j)}) P_{\bbe}(\bullet,y)](\bx) &= [\La ( \pi(\widetilde \sfY_{t,u}^{(j)}) M_u \widetilde v_{\by})](\bx) = \la_{t,y_j} P_{\bbe}(\bx,\by),\\
[\widetilde \rho(\sfY_{s,u}^{(j)}) P_{\bbe}(\bx,\bullet)](\by) &= [\widetilde \La ( \pi(\sfY_{s,u}^{(j)}) M_u v_{\bx} )](\by) = \la_{x_{N-j+1},s} P_{\bbe}(\bx,\by).
\end{split}
\]
Our goal is now to write these eigenvalue equations as explicit $q$-difference equations.

\begin{theorem} \label{thm:difference equation}
For $j=1,\ldots,N$ $\rho(\widetilde \sfY_{t,u}^{(j)})$ is the $q$-difference operator given by
\[
\rho(\widetilde \sfY_{t,u}^{(j)}) = \sum_{\bnu \in \{-1,0,1\}^j } V_{\bnu,\bbe}^{(j)}(\bx) \mathcal T_{\bnu} -\left( \frac{(u+u^{-1})\mu_{x_{j+1}}}{q^{-1}+q}-\mu_t\right)\mathrm{Id}
\]
where,
\[
V_{\bnu,\bbe}^{(j)}(\bx) =
V_{\bnu}^{(j)}(\bx)\left( \frac{(u+u^{-1})\mu_{q^{2\nu_1}x_1}}{q^{-1}+q} + \frac{(qu-q^{-1}u^{-1})(\mu_{q^{2\nu_1}x_1}-\mu_{x_1})}{(q^{-1}-q)(q^{-1}+q)}-\mu_t\right),
\]
with $V_{\bnu}^{(j)}$ given in Proposition \ref{prop:K-2 mult difference operator}. In particular, the multivariate Askey-Wilson polynomials $P_{\bbe}(\bx,\by)$ satisfy
\[
\sum_{\bnu \in \{-1,0,1\}^j } V_{\bnu,\bbe}^{(j)}(\bx) [\mathcal T_{\bnu} P_{\bbe}(\bullet,\by)](\bx) -\left( \frac{(u+u^{-1})\mu_{x_{j+1}}}{q^{-1}+q}-\mu_t\right) P_{\bbe}(\bx,\by)= \la_{t,y_j}\, P_{\bbe}(\bx,\by).
\]
\end{theorem}
\begin{proof}
First note that it follows from $\widetilde \sfY_{t,u}^{(j)} = \De^{j-1}(\widetilde Y_{t,u}) \tensor 1^{\tensor(N-j)}$ that $\rho(\widetilde \sfY_{t,u}^{(j)})$ acts only on the variables $x_1,\ldots,x_j$. So we may fix $x_{j+1},\ldots,x_N$, and consider only the action of $\rho_{\bk_j}(\De^{j-1}(\widetilde Y_{t,u}))$ on appropriate functions in $x_1,\ldots,x_j$.

Using Lemma \ref{lem:Ys S T} we can write $\De^{j-1}(\widetilde Y_{t,u})$ in terms of $\De^{j-1}(S)$ and $\De^{j-1}(T)$, where
\[
\begin{split}
\De^{j-1}(S) &= \De^{j-1}(K^{-2}) \De^{j-1}(Y_{x_{j+1}} + \mu_{x_{j+1}}1)-\mu_{x_{j+1}} \De^{j-1}(1),\\
\De^{j-1}(T) &=\frac{1}{(q-q^{-1})}\Big[\De^{j-1}(K^{-2})\De^{j-1}(Y_{x_{j+1}}+\mu_{x_{j+1}}1) - \De^{j-1}(Y_{x_{j+1}}+\mu_{x_{j+1}}1)\De^{j-1}(K^{-2}) \Big].
\end{split}
\]
Note that we use here that $S$ and $T$, which are defined in terms of $Y_s$, are actually independent of $s$. So we can conveniently replace $s$ by $x_{j+1}$. Now we use that $\rho_{\bk_j}(\De^{j-1}(Y_{x_{j+1}}+\mu_{x_{j+1}}1))$ is multiplication by $\mu_{x_1}$, and $\rho_{\bk_j}(\De^{j-1}(K^{-2}))$ is given as an explicit difference operator in Proposition \ref{prop:K-2 mult difference operator}. Recall here that $\sfK^{-2,(j)} = \De^{j-1}(K^{-2}) \tensor 1^{\tensor(N-j)}$. Then
\[
\rho_{\bk_j}(\De^{j-1}(S)) = \sum_{\bnu \in \{-1,0,1\}^j } \mu_{q^{2\nu_1}x_1} V_{\bnu}^{(j)}(\bx) \mathcal T_{\bnu} - \mu_{x_{j+1}}\mathrm{Id},
\]
and
\[
\rho_{\bk_j}(\De^{j-1}(T))=\frac{1}{q^{-1}-q} \sum_{\bnu \in \{-1,0,1\}^{j}}(\mu_{q^{2\nu_1} x_1}-\mu_{x_1})V_{\bnu}^{(j)}(\bx)\mathcal T_{\bnu} .
\]
This gives the following expression for $\rho_{\bk_j}(\De^{j-1}(\widetilde Y_{t,u}))$,
\[
\rho_{\bk_j}(\De^{j-1}(\widetilde Y_{t,u})) = \sum_{\bnu \in \{-1,0,1\}^j } V_{\bnu,\bbe}^{(j)}(\bx) \mathcal T_{\bnu} -\left( \frac{(u+u^{-1})\mu_{x_{j+1}}}{q^{-1}+q}-\mu_t\right)\mathrm{Id},
\]
with
\[
V_{\bnu,\bbe}^{(j)}(\bx) =
V_{\bnu}^{(j)}(\bx)\left( \frac{(u+u^{-1})\mu_{q^{2\nu_1}x_1}}{q^{-1}+q} + \frac{(qu-q^{-1}u^{-1})(\mu_{q^{2\nu_1}x_1}-\mu_{x_1})}{(q^{-1}-q)(q^{-1}+q)}-\mu_t\right). \qedhere
\]
\end{proof}
To compare the difference equations for the multivariate Askey-Wilson polynomials in Theorem \ref{thm:difference equation} with Iliev's difference equation \cite[Proposition 4.2]{Il}, let us write the coefficients $V_{\bnu,\bbe}$ in terms of the parameters $\al_0,\ldots,\al_{N+2}$ defined by \eqref{eq:parameter alpha}. We have
\[
V_{\bnu,\bbe}^{(j)} =  \frac{1}{q^{-1}-q}\prod_{i=0}^j V_{\bnu,\bbe,i}^{(j)}(\bx),
\]
where for $i=0$,
\[
V_{\bnu,\bbe,i}^{(j)} (\bx) =
\begin{cases}
-\dfrac{q \al_0}{\al_1} (1-\frac{\al_1}{\al_0^2}x_1^{\nu_1})(1-\al_1 q^{-2} x_1^{-\nu_1}), & \nu_1\neq 0,\\
\dfrac{(\al_0+\al_0^{-1})(x_1+x_1^{-1})}{q^{-1}+q}- \left(\dfrac{q \al_0}{\al_1} + \dfrac{\al_1}{q\al_0}\right), & \nu_1=0,
\end{cases}
\]
and for $i=1,\ldots,j$,
\[
V_{\bnu,\bbe,i}^{(j)} (\bx) =
\begin{cases} \dfrac{\al_{i}}{\al_{i+1}}\dfrac{(1-\frac{\al_{i+1}}{\al_i}x_{i+1} x_i^{\nu_i} ) ( 1-\frac{\al_{i+1}}{\al_i}\frac{x_i^{\nu_j}}{x_{i+1}} ) } {(1-x_i^{2\nu_j})(1-q^2x_i^{2\nu_j})}, & \nu_i \neq 0,\, \nu_{i+1}=0,\\ \\
\dfrac{q^2(q^{-1}+q)(\frac{q\al_{j}}{\al_{j+1}} + \frac{ \al_{j+1} }{q \al_j }) - q^2 (x_j+x_j^{-1})(x_{j+1} + x_{j+1}^{-1} )} {(1-q^2 x_j^2)(1-q^2x_j^{-2})}, & \nu_i=0, \, \nu_{i+1}=0,\\ \\
\dfrac{\al_{i}}{\al_{i+1}} \dfrac{(1-\frac{\al_{i+1}}{\al_i} x_{i+1}^{\nu_{i+1}} x_i^{\nu_i})  (1-\frac{\al_{i+1}}{\al_i} q^2 x_{i+1}^{\nu_{i+1}} x_i^{\nu_i}) }{(1-x_i^{2\nu_i})(1-q^2 x_i^{2\nu_i})}, & \nu_i\neq 0,\, \nu_{i+1}\neq 0,\\ \\
\dfrac{q^3 \al_i}{\al_{i+1}} \dfrac{(q^{-1}+q)(1-\frac{\al_{i+1}}{\al_i}x_{i+1}^{\nu_{i+1}} x_i) (1-\frac{\al_{i+1}}{\al_i}\frac{x_{i+1}^{\nu_{i+1}}}{x_i})}{(1-q^2 x_i^2)(1-q^2x_i^{-2})}, & \nu_i=0,\,\nu_{i+1} \neq 0,
\end{cases}
\]
with the assumption $\nu_{j+1}=0$. With these expressions and Theorem \ref{thm:Pbeta=multivariate AWpol} it is a straightforward calculation to show that the difference equations we obtained are equivalent to Iliev's difference equations.

An explicit expression for the difference operators $\widetilde \rho(\sfY_{s,u}^{(j)})$ is obtained in the same way as in Theorem \ref{thm:difference equation}. This gives explicit recurrence relations for the Askey-Wilson polynomials $P_{\bbe}(\bx,\by)$. We just state the result here.
\begin{theorem}
For $j=1,\ldots,N$ $\widetilde \rho(\sfY_{t,u}^{(j)})$ is the $q$-difference operator given by
\[
\widetilde \rho(\sfY_{s,u}^{(j)}) = \sum_{\bnu \in \{-1,0,1\}^j } \widetilde V_{\bnu,\bbe}^{(j)}(\by) \hat{\mathcal T}_{\bnu} +\left( \frac{(u+u^{-1})\mu_{y_{N-j}}}{q^{-1}+q}-\mu_s\right)\mathrm{Id}
\]
where
\[
\widetilde V_{\bnu,\bbe}^{(j)}(\by)= -\widetilde V_{\bnu}^{(j)}(\by)\left( \frac{(u+u^{-1})\mu_{q^{-2\nu_N}y_N}}{q^{-1}+q} + \frac{(qu^{-1}-q^{-1}u)(\mu_{q^{-2\nu_N}y_N}-\mu_{y_N})}{(q^{-1}-q)(q^{-1}+q)}-\mu_s\right),
\]
with $\widetilde V_{\bnu}^{(j)}(\by)$ given in Proposition \ref{prop:K2 mult difference operator}.
In particular, the multivariate Askey-Wilson polynomials $P_{\bbe}(\bx,\by)$ satisfy
\[
\sum_{\bnu \in \{-1,0,1\}^j } \widetilde V_{\bnu,\widetilde \bbe}^{(j)}(\by) [\hat{\mathcal T}_{\bnu} P_{\bbe}(\bx,\bullet)](\by) +\left( \frac{(u+u^{-1})\mu_{y_{N-j}}}{q^{-1}+q}-\mu_s\right) P_{\bbe}(\bx,\by)= \la_{x_{N-j+1},s}\, P_{\bbe}(\bx,\by),
\]
for $\by \in S$.
\end{theorem}
Note that $\widetilde V_{\bnu,\bbe}^{(j)}(\by) = V_{\bnu,\widetilde \bbe}^{(j)}(\hat \by)$, where
$\widetilde \bbe = (t,s,u,k_N,\ldots,k_1,q^{-1})$.

\section{Appendix}

\subsection{Convergence of the sum for $P_{\be}(x,y)$}
The function $P_{\be}(x,y)$ is defined in \eqref{eq:def univariate AW} by
\begin{equation} \label{eq:sum appendix}
P_{\be}(x,y)  =\langle M_u \widetilde v_{y,t}, v_{x,s} \rangle_H =  \sum_{n =0}^\infty \om(n) v_{x,s}(n) \widetilde v_{y,t}(n) u^n,
\end{equation}
where $y= tq^{-k-2m} \in S$ and $x \in \T$.
The eigenfunctions $v_{x,s}(n)$ and $\widetilde v_{y,t}(n)$, see \eqref{eq:definition vxs} and \eqref{eq:definition tilde vxs},  are Al-Salam--Chihara polynomials in base $q^2$ and $q^{-2}$, respectively. Using the expressions for these polynomials as $q$-hypergeometric functions we show here that the sum for $P_\be(x,y)$ converges.

We start with $v_{x,s}(n)$. Applying transformation formulas \cite[(III.2)]{GR} and then \cite[(III.31),(III.1)]{GR} gives
\[
\begin{split}
v_{x,s}(n) &= \ga_n x^n(q^k/sx;q^2)_n \rphis{2}{1}{q^{-2n}, q^ksx}{q^{2-2n-k}sx}{q^2,\frac{sq^{2-k}}{x}} \\
&= \ga_n x^n \frac{(q^k/sx;q^2)_n (sq^k/x;q^2)_\infty}{ (sq^{2-k}/x;q^2)_\infty} \rphis{2}{1}{q^{2-k}sx, q^{2-2k-2n} }{q^{2-2n-k}sx}{q^2, \frac{sq^k}{x}} \\
& = \frac{\ga_nc(x) x^n}{(q^{2k+2n};q^2)_\infty} \rphis{2}{1}{q^ksx, q^kx/s}{q^2x^2}{q^2,q^{2+2n}} + \mathrm{idem}(x \leftrightarrow x^{-1})
\end{split}
\]
where
\[
\begin{split}
\ga_n &= \frac{q^{-n(k-1)/2}}{ (q^2;q^2)_n},\\
c(x) &= \frac{ (q^ks/x, q^k/sx;q^2)_\infty}{(1/x^2;q^2)_\infty}.
\end{split}
\]
So for $n \to \infty$
\begin{equation} \label{eq:asymptotics vxs}
v_{x,s}(n) \sim
C\,q^{-n(k-1)/2} \Big( c(x) x^n + c(x^{-1}) x^{-n} \Big), \qquad  x \in \T,
\end{equation}
where $C$ is a constant independent of $n$.

Next we consider $\widetilde v_{y,t}(n)$. This function can be written as
\[
\begin{split}
\widetilde v_{tq^{-k-2m},t}(n) = \widetilde \ga_n \widetilde c_m \rphis{2}{1}{q^{-2m},  q^{2m+2k}/t^2 }{q^2/t^2}{q^2,q^{2n+2}},
\end{split}
\]
with
\[
\begin{split}
\widetilde \ga_n &= t^{-n} q^{n(3-k)/2}\frac{(q^{2k};q^{2})_n }{ (q^{2};q^{2})_n },\\
\widetilde c_m & = (-1)^m t^{2m} q^{-m(m+1)} \frac{ (q^2/t^2;q^2)_m }{(q^{2k};q^2)_m }.
\end{split}
\]
Then for $n \to \infty$,
\begin{equation} \label{eq:asymptotics tilde vyt}
\widetilde v_{tq^{-k-2m},t}(n) \sim  C t^{-n} q^{n(3-k)/2},
\end{equation}
where $C$ is independent of $n$.

The weight function $\om$ satisfies
\[
\om(n) = q^{n(k-1)} \frac{(q^2;q^2)_n }{ (q^{2k};q^2)_n } = \mathcal O(q^{n(k-1)}), \qquad n \to \infty,
\]
Then from \eqref{eq:asymptotics vxs} and \eqref{eq:asymptotics tilde vyt} it follows that the summand in \eqref{eq:sum appendix} is of order $\mathcal O((q/t)^n)$, so that the sum \eqref{eq:sum appendix} converges absolutely (recall that $|t|\geq q^{-1}$ and $u \in \T$).

Furthermore, for $x=sq^{k+2m}$, $m \in \N$, we have $c(x)=0$, so that
\[
v_{x,s}(n) \sim C c(x^{-1}) x^{-n}, \qquad n \to \infty.
\]
We see that in this case the sum \eqref{eq:sum appendix} converges absolutely if $|t|>q^{k+2m+1}$.

\subsection{Evaluation of $P_\be(x,tq^{-k})$}
By inserting in \eqref{eq:sum appendix} explicit expressions for $v_{x,s}(n)$ and $\widetilde v_{tq^{-k},t}(n)$ in terms of $q$-hypergeometric functions and using $x \leftrightarrow x^{-1}$ invariance, we obtain
\begin{equation} \label{eq:P=sum}
\begin{split}
P_{\be}(x,tq^{-k}) & = \sum_{n=0}^\infty \left(\frac{qu}{x t} \right)^n\frac{(q^kx/s;q^2)_n}{(q^2;q^2)_n}
\rphis{2}{1}{q^{-2n},q^ks/x}{q^{2-2n-k}s/x}{q^2,sxq^{2-k}}.
\end{split}
\end{equation}
We write the $_2\varphi_1$ function as a sum and interchange the order of summation
\[
\sum_{n=0}^\infty \sum_{m=0}^n A_{n,m}= \sum_{m=0}^\infty \sum_{n=m}^\infty A_{n,m}.
\]
Then replace $n-m=l$ and use the identity
\[
\frac{(q^{-2l-2m};q^2)_m}{ (sq^{2-k-2l-2m}/x;q^2)_m} = \frac{(q^kx/s;q^2)_l (q^2;q^2)_{l+m}}{ (q^2;q^2)_l (q^kx/s;q^2)_{l+m}} \left( \frac{ xq^{k-2} }{s}\right)^m
\]
to find
\[
P_{\be}(x,tq^{-k}) = \sum_{m=0}^\infty \frac{ (q^k s/x;q^2)_m }{(q^2;q^2)_m } \left( \frac{qux}{t}\right)^m \sum_{l=0}^\infty \frac{ (q^kx/s;q^2)_l }{ (q^2;q^2)_l } \left( \frac{ qu}{xt}\right)^l.
\]
Recall that $s,u,x \in \T$ and $|t|\geq q^{-1}$, then we see that both sums converge. They can be evaluated using the $q$-binomial formula \cite[(II.3)]{GR}, leading to
\[
P_{\be}(x,tq^{-k}) = \frac{ (q^{k+1}us^{\pm 1}/t;q^2)_\infty }{ (qux^{\pm 1}/t;q^2)_\infty }.
\]

\subsection{Overview of various Hilbert spaces}
\*\\
In Sections \ref{sect:Al-Salam--Chihara} and \ref{sect:AW polynomials} we use the following Hilbert spaces of univariate complex-valued functions:
\begin{itemize}
\item $H=H_k$ is the representation space of $\pi_k$. It is the Hilbert space of functions on $\N$ with inner product
\[
\begin{split}
\langle f,g\rangle_H &= \sum_{n \in \N} f(n) \overline{g(n)} \om(n), \\
\om(n) &= \om_k(n)= q^{n(k-1)} \frac{ (q^2;q^2)_n }{(q^{2k};q^2)_n}.
\end{split}
\]

\item $\mathcal H=\mathcal H_{k,s}$ is the representation space of $\rho_{k,s}$. It is the Hilbert space of function on $\T$ which are invariant under $x \leftrightarrow x^{-1}$, with inner product
\[
\begin{split}
\langle f,g\rangle_{\mathcal H} &= \frac1{4\pi i} \int_\T f(x) \overline{g(x)} w(x) \frac{dx}{x}, \\
w(x)&=w_{k,s}(x) =\frac{ (q^2,q^{2k},x^{\pm 2};q^2)_\infty}{(q^k s^{\pm 1} x^{\pm 1};q^2)_\infty}.
\end{split}
\]
\item $\widetilde {\mathcal H} = \widetilde{\mathcal H}_{k,t}$ is the representation space of $\widetilde \rho_{k,t}$. It is the Hilbert space of functions on $S=S_{k,t,q} = \{ tq^{-k-2m} \mid m \in \N\}$, with inner product
\[
\begin{split}
\langle f,g\rangle_{\widetilde{\mathcal H}} &= \sum_{y \in S} f(y) \overline{g(y)}\, \widetilde w(y), \\
\widetilde w(y) &= \widetilde w_{k,t}(y) = \frac{ 1-q^{4m+2k} /t^2 }{1-q^{2k}/t^{2}} \frac{ (q^{2k}/t^2,q^{2k};q^2)_m (q^{2m+2}/t^2;q^2)_\infty }{(q^2;q^2)_m (q^{2k+2}/t^2;q^2)_\infty} t^{-2m}q^{2m^2},
\end{split}
\]
for $y = tq^{-k-2m},  m\in \N$.
\end{itemize}

\bigskip

In Sections \ref{sect:multivariate ASC pol} and \ref{sect:multivariate AW pol} we use the following Hilbert spaces of multivariate complex-valued functions:
\begin{itemize}
\item $H=H_{\bk}$ is the representation space of $\pi_\bk$. It is the Hilbert space of functions on $\N^N$ with inner product
\[
\begin{split}
\langle f,g\rangle_{H} &= \sum_{\bn \in \N^N} f(\bn) \overline{g(\bn)} \, \om(\bn), \\
\om(\bn) &= \om_{\bk}(\bn) = \prod_{j=1}^N \om_{k_j}(n_j) = \prod_{j=1}^N \frac{ (q^2;q^2)_{n_j} }{(q^{2k_j};q^2)_{n_j}} q^{n_j(k_j-1)}.
\end{split}
\]
\item $\mathcal H=\mathcal H_{\bk,s}$ is the representation space of $\rho_{\bk,s}$. It is the Hilbert space of functions on $\T^N$ which are invariant under $x_j \leftrightarrow x_j^{-1}$, $j=1,\ldots,N$, with inner product
\[
\begin{split}
\langle f,g \rangle_{\mathcal H} &= \int_{\T^N} f(\bx) \overline{g(\bx)} \, w(\bx)\, \frac{d\bx}{\bx}, \\
w(\bx)& = w_{\bk,s}(\bx)= \prod_{j=1}^N w_{k_j,x_{j+1}}(x_{j}) = \prod_{j=1}^N \frac{ (q^2,q^{2k_j},x_j^{\pm 2};q^2)_\infty}{(q^{k_j}x_{j+1}^{\pm 1} x_j^{\pm 1};q^2)_\infty}.
\end{split}
\]
\item $\widetilde{\mathcal H}= \widetilde{\mathcal H}_{\bk,t}$ is the representation space of $\widetilde \rho_{\bk,t}$. It is the Hilbert space of functions on \mbox{$S=S_{\bk,t,q} = \Big\{ tq^{-\Sigma(\bk)-2\Sigma(\bm)} \mid \bm \in \N^N \Big\}$}, with inner product
\[
\begin{split}
\langle f,g\rangle_{\widetilde{\mathcal H}} &= \sum_{\by \in S} f(\by) \overline{g(\by)}\, \widetilde w(\by),\\
\widetilde w(\by) &= \prod_{j=1}^N \widetilde w_{k_j,y_{j-1}}(y_j) =
 \prod_{j=1}^N \left(\frac{ q^{2K_{j-1}+4M_{j-1}+2m_j}}{t^2} \right)^{m_j}
 \frac{1-q^{2K_j+4M_j}/t^2 }{ 1-q^{2K_j+4M_{j-1}}/t^2} \\
  & \qquad \times \frac{ (q^{2K_{j}+4M_{j-1}}/t^2,q^{2k_j};q^2)_{m_j} (q^{2K_{j-1}+4M_{j-1}+2m_j+2}/t^2;q^2)_\infty }{(q^2;q^2)_{m_j} (q^{2K_{j}+4M_{j-1}+2}/t^2;q^2)_\infty},
\end{split}
\]
where $\by=tq^{-\Sigma(\bk)-2\Sigma(\bm)}$, $M_j=\Sigma(\bm)_j = \sum_{i=1}^j m_i$, $K_j = \Sigma(\bk)_j = \sum_{i=1}^j k_i$ and \mbox{$M_0=0=K_0$}.
\end{itemize}

\end{document}